\documentclass{amsart}

\usepackage[T1]{fontenc}
\usepackage[utf8]{inputenc}
\usepackage[english]{babel}

\usepackage{hyperref}

\usepackage{amssymb}
\usepackage{amsthm}
\usepackage{amstext}
\usepackage{amsfonts} 					
\usepackage{amsmath}

\usepackage{enumitem} 

\usepackage{tikz}						
\usetikzlibrary{matrix}

\newtheorem{thm}{Theorem}[section]
\newtheorem{lemma}[thm]{Lemma}
\newtheorem{prop}[thm]{Proposition}
\newtheorem{cor}[thm]{Corollary}
\newtheorem{claim}[thm]{Claim}
\newtheorem{ex}[thm]{Example}
\newtheorem{rmk}[thm]{Remark}

\newcommand{\mb}{\mathbb}

\newcommand{\C}{\mb C}
\newcommand{\Pj}{\mb P}

\newcommand{\mf}{\mathfrak}
\newcommand{\g}{\mf g}
\newcommand{\h}{\mf h}

\newcommand{\m}{\mf m}

\newcommand{\mc}{\mathcal}

\newcommand{\F}{\mc F}
\newcommand{\G}{\mc G}
\newcommand{\mcE}{\mc E}
\newcommand{\mcL}{\mc L}

\newcommand{\mbf}{\mathbf }

\DeclareMathOperator{\mrrestr}{restr}

\newcommand{\hatnabla}{\widehat{\nabla}}
\newcommand{\de}{\partial}

\newcommand\restr[2]{{
  \left.\kern-\nulldelimiterspace 
  #1 
  \vphantom{\big|} 
  \right|_{#2} 
  }}

\newcommand{\jetX}{\mc{J}_X}
\newcommand{\jetXF}{\mc{J}_{X/\F}}
\newcommand{\jetY}{\mc{J}_Y}

\newcommand{\prolong}[2]{#1^{(#2)}}
\newcommand{\prolongPjntwo}{\prolong{(\Pj^n)}{2}}
\newcommand{\prolongPjqtwo}{\prolong{(\Pj^q)}{2}}

\newcommand{\germe}{\mathcal{O}}
\newcommand{\mcHom}{\mathcal{H}om}

\DeclareMathOperator{\codim}{codim}
\DeclareMathOperator{\Spec}{Spec}
\DeclareMathOperator{\Sym}{Sym}

\DeclareMathOperator{\Aff}{Aff}
\DeclareMathOperator{\PSL}{PSL}
\newcommand{\psl}{\mf{psl}}

\DeclareMathOperator{\Ad}{Ad}
\DeclareMathOperator{\Aut}{Aut}

\DeclareMathOperator{\GL}{GL}
\DeclareMathOperator{\Id}{Id}
\DeclareMathOperator{\tr}{tr}

\newcommand{\ideal}[1]{\left< #1\right>}

\newcommand{\CNF}{{N^*_\F}}
\newcommand{\CTF}{\Omega_\F}

\DeclareMathOperator{\sing}{sing}

\newcommand{\df}{d_{\mc F}}

\title{Jets of flat partial connections}
\author[G. Fazoli]{Gabriel Fazoli}
\address{IMPA, Estrada Dona  Castorina, 110, 22460-320, Rio de Janeiro, RJ, Brazil}
\curraddr{IMECC - UNICAMP, Rua S\'ergio Buarque de Holanda, 651,13083-859 Campinas, SP, Brazil}
\email{gabrielfazoli@gmail.com}

\keywords{Holomorphic foliations, partial connections, jets; transverse affine structures, transverse projective structures}

\subjclass[2020]{Primary 32M25; Secondary 53C05, 53C12}

\date{}

\begin{document}

\begin{abstract}
    We define and study jets of flat partial connections in the setting of smooth foliations and flat partial connections on locally free sheaves. In the case of codimension one foliations, we apply this definition to characterize transverse affine and transverse projective structures. For foliations of arbitrary codimension, we use jets of the Bott connection on the normal sheaf to define the prolongation of a transverse projective structure, and then apply it to produce singular transverse projective structures.
\end{abstract}

\maketitle

\section{Introduction}

\subsection*{Partial connections}

Let $\F$ be a foliation on a complex manifold $X$, with tangent sheaf $T_{\F}$. Given a coherent $\germe_X$-module $\mcE$, a $\F$-\emph{partial connection} on $\mcE$ is an $\germe_X$-linear morphism $\nabla: T_{\F} \rightarrow \mathrm{End}_{\C}(\mcE)$,  assigning to each $v\in T_{\F}$ a $\C$-linear morphism $\nabla_v$ satisfying the Leibniz rule (that is, a \emph{covariant differential operator}). In this sense, a partial connection  gives rise to a derivation of sections of $\mcE$ along directions tangent to the leaves of the foliation. If, in addition, $\nabla$ is also a morphism of Lie algebroids, we say that $\nabla$ is flat.

Flat partial connections were first employed in holomorphic foliation theory by P. Baum and R. Bott (see \cite{baum-bott-70-zbMATH03308090,baum-bott-72-zbMATH03423310,bott-67-zbMATH03286795,bott-67-zbMATH03235065}), who used the \emph{Bott connection} to develop a \emph{residue theory} for the singular set of a foliation (see \cite[Theorem 2]{baum-bott-72-zbMATH03423310}). Other important works with a similar goal are \cite{abate-bracci-filippo-suwa-tovena-13-zbMATH06176462,seade-suwa-96-zbMATH00892565,suwa-98-zbMATH01194470}. Flat partial connections can also be used to describe  \emph{transverse structures} associated with a foliation. Works adopting this perspective are \cite{biswas-02-zbMATH01786109, cousin-pereira-14-zbMATH06399599, loray-pereira-touzet-16-zbMATH06670708}.  We also note the use of flat partial connections to study \emph{foliated structures} in \cite{deroin-guillot-23-zbMATH07687688}.

In this work, we adopt the second point of view and aim to apply the theory of flat partial connections to the study of transverse structures. Our strategy is based on works that use the theory of jet bundles, frame bundles, and differential equations to investigate special structures on varieties (see \cite{Deligne-70-zbMATH03385791,gunning-67-zbMATH03233043,kobayashi-nagano-64-zbMATH03191725,loray-marin-09-zbMATH05718008}). 

More precisely, our first main contribution is a definition of \emph{jets of flat partial connections} on locally free sheaves for smooth foliations (see Subsection \ref{Subsection: transverse jets}). Namely, starting with a flat partial connection $\nabla$ on a locally free sheaf $\mcE$, we define the \emph{$k$-th sheaf of transverse jets of $(\mcE,\nabla)$}, denoted by $\jetXF^k(\nabla)$, endowed with a flat partial connection $\nabla^k$, such that the $k$-th jet of a flat section of $\nabla$ is a flat section of $\nabla^k$.

Jets of flat partial connections for foliations of codimension one were previously considered in \cite{biswas-02-zbMATH01786109}. Here we give a sheaf-theoretic construction for smooth holomorphic foliations of arbitrary codimension and use it in the applications below.

\subsection*{Transverse structures for codimension one foliations} 

In \cite{gunning-67-zbMATH03233043}, R. Gunning provided a description of affine and projective structures on Riemann surfaces in terms of special differential operators. In the affine case, this description can be further translated in terms of a connection on the tangent bundle (see \cite[Lemma 1]{gunning-67-zbMATH03233043}). Moreover, using the work of M. Atiyah (see \cite[Theorem 5]{atiyah-57-zbMATH03128044}), one concludes that affine structures are naturally in bijection with splittings of the short exact sequence of the sheaf of first jets of sections of $T_X$.

In the case of a projective structure, P. Deligne, still building on the work of Gunning, provided a description in terms of connections on first jets and second-order differential equations (see \cite[Proposition 5.12]{Deligne-70-zbMATH03385791}).

For codimension one smooth foliations, we study transverse affine and transverse projective structures. In this setting, we generalize the results mentioned above by giving a characterization of transverse affine structures (see Corollary \ref{C: transversely affine structures}) and transverse projective structures (see Theorem \ref{T: second order transverse differential equations}).

\subsection*{Singular \texorpdfstring{transverse}{transverse} projective structures}

Given a codimension $q$ foliation $\F$, the definition of a transverse projective structure for $\F$ as a foliated atlas whose coordinate changes are automorphisms of $\Pj^q$ is only appropriate when $\F$ is smooth. In order to study transverse structures for singular foliations, we consider instead \emph{singular transverse projective structures}. In the holomorphic foliation literature, such structures can be defined in terms of collections of meromorphic $\psl(q+1,\C)$-forms with some compatibility equation, in terms of global meromorphic $\psl(q+1,\C)$-forms, or even as a $\Pj^q$-bundle equipped with a generically transverse foliation (for a discussion of the different definitions, see \cite{cerveau-lins-neto-loray-pereira-touzet-07-zbMATH05202834, loray-pereira-touzet-16-zbMATH06670708}).

In this context, our contribution is a construction we call the \emph{prolongation of a transverse projective structure} (see Subsection \ref{Subsection: Prolongation of transversely projective structures}). Namely, starting with a smooth foliation and a transverse projective structure, we construct a singular transverse Lie structure for the foliation induced by the flat partial connection $\nabla_B^1$ on the total space of $\jetXF^1(\nabla_B)$, where $\nabla_B$ stands for the Bott connection on $N_{\F}$. The main consequence of this construction is Theorem \ref{T: prolongation of transversely projective structure}, which shows how the prolongation can be used to produce singular transverse projective structures defined by a single global meromorphic $\psl(q+1,\C)$-form. 

\subsection*{Acknowledgements}

I am deeply grateful to Jorge Vitório Pereira for his guidance during my PhD, throughout which I studied the topics developed in this work. I am also grateful to João Pedro dos Santos for insightful discussions, references, and suggestions regarding notation. I thank Caio Melo for pointing out some typos in a preliminary version of this manuscript. The author acknowledges the support of FAPERJ (Grant number E26/201.353/2023)  and FAPESP (Grant number 26/14134-9).

\subsection*{Structure of the paper} In Section \ref{Section: Foliations}, we recall the main definitions from foliation theory and introduce the definitions of transverse structures needed in this work. In Section \ref{Section: Partial Connections}, we establish the basic theory of partial connections. Section \ref{Section: jets} begins with a review of the theory of jets of sections of sheaves, followed by the definition of jets of flat partial connections and a discussion of properties relevant to our applications. In Section \ref{Section: transverse differential equations}, we define transverse differential equations and apply the theory developed here to the problem of flat extensions of flat partial connections (see Theorem \ref{T: flat meromorphic connections, bijection}). We then provide characterizations of transverse affine structures (see Corollary \ref{C: transversely affine structures}) and transverse projective structures (see Theorem \ref{T: second order transverse differential equations}). Finally, in Section \ref{Section: prolongation}, we construct the prolongation of transverse projective structures (see Lemma \ref{L: psl structure of the prolongation}) and present a theorem relating the prolonged structure to singular transverse projective structures (Theorem \ref{T: prolongation of transversely projective structure}).

\subsection*{Use of Large Language Models}
The initial version of this manuscript was prepared without LLM. In preparing this revised version, the author used LLM only for language editing and feedback on the exposition and mathematical arguments. The author takes full responsibility for the manuscript.

\section{Foliations}\label{Section: Foliations}

\subsection{Foliations}  A \emph{foliation} $\F$ on a complex manifold $X$ is determined by a saturated and involutive coherent subsheaf $T_{\F} \subset T_X$, called the \emph{tangent sheaf} of $\F$.  The \emph{dimension} of $\F$ is the rank of $T_{\F}$. The \emph{cotangent sheaf} of $\F$ is defined by $\CTF^1 := T_{\F}^*$.

The inclusion of $T_{\F}$ into $T_X$ induces the short exact sequence
\begin{equation}\label{Eq: short exact sequence of the tangent sheaf}
    0 \rightarrow T_{\F} \rightarrow T_X \rightarrow \frac{T_X}{T_{\F}} \rightarrow 0,
\end{equation}
which we will refer to as the \emph{exact sequence of the tangent sheaf}. The morphism $\Omega_X^1 \rightarrow \CTF^1$ defined as the dual of the inclusion $T_{\F} \rightarrow T_X$ is called the \emph{restriction morphism}. We define the \emph{conormal sheaf} of $\F$, denoted by $\CNF$, as the kernel of the restriction morphism, that is, 
\[
\CNF := \{\omega \in \Omega_X^1 ; \omega(v)=0,\forall v\in T_{\F} \},
\]
and by definition it is isomorphic to $(T_X/T_{\F})^*$. Additionally, the definition of $\CNF$ leads to the exact sequence
\begin{equation}\label{Eq: short exact sequence of the conormal sheaf}
    0 \rightarrow \CNF \rightarrow \Omega_X^1 \rightarrow \CTF^1,
\end{equation}
which we refer to as the \emph{exact sequence of the conormal sheaf}. This sequence is left exact in general, and it is short exact when $\F$ is smooth. The \emph{normal sheaf} of $\F$ is defined to be $N_{\F} := (\CNF)^*$. Finally, the \emph{codimension} of $\F$ is the rank of $\CNF$.

\subsection{Singular and smooth loci}
The \emph{singular locus} of $\F$ is the set of points $p\in X$ where the quotient $T_X/T_{\F}$ is not locally free, and it is denoted by $\sing(\F)$. The singular locus is always a closed subvariety of $X$, and since $T_X/T_{\F}$ is torsion-free, it follows that $\sing(\F)$ has codimension at least two. 

The \emph{smooth locus} of $\F$ is the complement $X-\sing(\F)$, that is, the set of \emph{smooth} points of the foliation $\F$. Let $q = \codim(\F)$. By the \emph{Frobenius Theorem}, for every smooth point $x\in X$ of the foliation, there exists a submersion $\phi: U \rightarrow \C^{q}$ defined in a neighborhood of $x$ such that 
\[
    \restr{T_{\F}}{U} = \ker(d\phi: T_U \rightarrow \phi^*T_{\C^{q}}),
\]
that is, $\restr{T_{\F}}{U}$ is the relative tangent bundle of the submersion $\phi$. We say that $\phi$ is a \emph{foliated chart} for $\F$. Concretely, this is the same as saying that there exists a system of coordinates $(x_1,\ldots,x_q,x_{q+1},\ldots x_n)$ on a neighborhood of $x\in X$ such that $T_{\F,x}$ is the free $\germe_{X,x}$-module generated by the vectors $\{\de/ \de x_{q+1}, \ldots \de / \de x_{n} \}$. We refer to a system of coordinates as above as a \emph{foliated system of coordinates}. 

We say that $\F$ is smooth if $\sing(\F)=\emptyset$. In this case, the set of foliated charts $\mc C = \{\phi: U \rightarrow \C^{q}  \}$ of the foliation $\F$ defines a \emph{foliated atlas} for $\F$.

\subsection{\texorpdfstring{transverse}{Transverse} homogeneous structures} All Lie groups in this work are assumed to be complex. Let $\F$ be a smooth codimension $q$ foliation on $X$. Let $G$ be a Lie group, and let $H\subset G$ be a closed complex subgroup such that $\dim G/H =q$. A \emph{transverse homogeneous structure} for $\F$ is a collection of submersions $\mc C = \{ \phi_i: U_i \rightarrow G/H \}$ such that $\mc U = \{U_i\}$ is an open covering of $X$, $\phi_i$ determines $\F$ on $U_i$, and for every $x\in U_i\cap U_j$, there is a neighborhood $W\subset U_i\cap U_j$ of $x$ and an element $g\in G$ such that $\phi_i|_W=L_g\circ\phi_j|_W$. In the particular case where $H=\{e\}$, the transverse structure is called a \emph{transverse Lie structure}. We call $\mc C$ an \emph{atlas} for the transverse structure and identify the structure with the maximal atlas generated by $\mc C$. We refer to \cite[Chapter III]{Godbillon-91-zbMATH00050350} for a detailed discussion about transverse structures. 

For codimension one foliations, we have essentially three possible transverse homogeneous structures (see \cite[Lemma 1.8]{cerveau-lins-neto-loray-pereira-touzet-06-zbMATH05033717}): transverse \emph{euclidean} structures ($G=\C$, the group of translations on the complex line, and $H=\{e\}$, the trivial subgroup), transverse \emph{affine} structures ($G=\Aff(\C)$, the group of affine transformations, and $H=\C^*$, the subgroup of transformations fixing the origin) and transverse \emph{projective} structures ($G=\PSL(2,\C)$, the automorphism group of the projective line, and $H=G_p$, the isotropy subgroup of some point $p\in \Pj^1$).

\subsection{Singular \texorpdfstring{transverse}{transverse} homogeneous structures}

Observe that a transverse homogeneous structure for a foliation $\F$ induces a foliated atlas, and thus this definition applies only for smooth foliations. In order to study transverse structures for singular foliations, we need to consider more general notions of structures, such as the singular homogeneous structures introduced below.

Let $G$ be a Lie group, and let $H\subset G$ be a closed complex subgroup such that $\dim G/H =q$. Let $\g$ and $\h$ be their respective Lie algebras. A \emph{singular transverse homogeneous structure} for $\F$ is a collection of $\g$-valued 1-forms $\mc C = \{\Omega_i: T_{U_i} \rightarrow \g \otimes \germe_{U_i}(D)\}$ such that

\begin{enumerate}[label = (\roman*)]
    \item $\mc U = \{ U_i\}$ is a covering of $X$ and $D\ge 0$ is a divisor on $X$;
    \item\label{Item: g-valued 1-forms that are flat and induces the foliation} for every $i$, the $\g$-valued 1-form $\Omega_i$ is \emph{flat}, that is, 
    \[
        d\Omega_i + 1/2[\Omega_i,\Omega_i]=0,
    \]
    and the kernel of the induced morphism $T_{U_i} \rightarrow \g/\h \otimes \germe_{U_i}(D)$ is $\restr{T_{\F}}{U_i}$; and
    \item\label{Item: compatibility of local g structures} for every pair $(i,j)$ with $U_i\cap U_j\neq\emptyset$, there exists a map $h_{ij}:U_i\cap U_j\rightarrow H$ such that
    \begin{equation*}\label{Eq: compatibility of local affine structures}
         \Omega_i = \Ad(h_{ij}^{-1}) \circ \Omega_j + h_{ij}^*(\Omega_{H}),
    \end{equation*}
    where $\Ad: G \rightarrow \Aut(\g)$ is the adjoint representation of $G$, and $\Omega_{H}$ is the Maurer--Cartan form of the group $H$ (see \cite[Chapter 3, Definition 1.3]{sharpe-97-zbMATH00914851}). We will denote the \emph{compatibility equation} above simply by $\Omega_i \Rightarrow_{h_{ij}} \Omega_j$. 
\end{enumerate}
We refer to \cite[Chapter 3]{sharpe-97-zbMATH00914851} for the basic properties of the Maurer--Cartan form of Lie groups.  As before, in the particular case where $H = \{e \}$, the transverse structure is called a \emph{singular transverse Lie structure}.

\begin{rmk}
    A singular transverse homogeneous structure is the \emph{singular and transverse} counterpart of a \emph{flat Cartan atlas} for a complex manifold (see \cite[Chapter 5, Definitions 1.3 and 1.10]{sharpe-97-zbMATH00914851}), and much of the theory of flat Cartan connections can be translated to this context. An exposition of Cartan connections can be found in \cite{sharpe-97-zbMATH00914851}.
\end{rmk}

In this work, we consider the cases $G=\PSL(n+1,\C)$, the automorphism group of $\Pj^n$, with either $H=\{e\}$ (giving singular transverse Lie structures) or $H=G_p$, the stabilizer of a point $p\in\Pj^n$ (giving singular transverse projective structures). In this setting, $\psl(n+1,\C)=\mf{sl}(n+1,\C)$. The Maurer--Cartan form and structure equations for the projective model are described in \cite[Example IV.4.1]{kobayashi-95-zbMATH00758273}. We also consider codimension-one transverse affine structures, which are modeled by $G=\Aff(\C)$ and $H=\C^*$. These will be studied mainly through flat extensions of the Bott connection (see Proposition \ref{prop: transversely affine}). We omit the word ``singular'' when no confusion can arise.

\subsection{Primitives of singular \texorpdfstring{transverse}{transverse} homogeneous structures}

Fixing a holomorphic $\g$-valued form $\Omega$ on an open subset $U\subset X$, we say that a map $\Phi:U \rightarrow G$ is a \emph{primitive} of $\Omega$ if $\Omega = \Phi^*\Omega_G$. The condition for the existence of local primitives of $\g$-valued 1-forms is exactly the flatness of $\Omega$, as explained in \cite[Chapter 3, Sections 5 and 6]{sharpe-97-zbMATH00914851}. Moreover, given two primitives $\Phi_1,\Phi_2$ for $\Omega$, there exists $g\in G$ such that $\Phi_1 = L_g\circ \Phi_2$ (see \cite[Theorem 5.2]{sharpe-97-zbMATH00914851}). 

Given a transverse homogeneous structure $\mc C$, we say that $\Phi: U\rightarrow G$ is a primitive of $\mc C$ if it is a primitive of one of its $\g$-valued 1-forms $\Omega \in \mc C$. For every primitive $\Phi: U \rightarrow G$, we consider the induced map $\phi: U \rightarrow G/H$ given by the composition of $\Phi$ with the projection $G\rightarrow G/H$. Using the commutative diagram of the tangent bundle of a Klein Geometry (see \cite[Chapter 4, Section 5]{sharpe-97-zbMATH00914851}), it is easy to see that $\phi$ defines $\F$ on $U$.

Now let $\F$ be smooth, and suppose that it admits a (smooth) transverse homogeneous structure $\mc C = \{\phi_i: U_i\rightarrow G/H \}$ and a singular transverse homogeneous structure $\mc C' = \{\Omega_i: T_{U_i} \rightarrow \g \otimes \germe_{U_i}(D)\}$. We say that $\mc C$ and $\mc C'$ are equivalent if, for any primitive $\Phi: U \rightarrow G$ of $\mc C'$, the induced map $\phi: U \rightarrow G/H$ belongs to $\mc C$.

\section{Partial connections}\label{Section: Partial Connections}

Throughout this section, $X$ is a complex manifold and $\F$ is a possibly singular foliation on $X$. We assume that $\F$ is smooth whenever that hypothesis is needed.

\subsection{Definition}\label{Subsection: definition} Let $\mcE$ be a coherent $\germe_X$-module. Let us denote by $\CTF^1(\mcE) := \mcHom_{\germe_X}(T_{\F},\mcE)$ the sheaf of foliated differential 1-forms with coefficients in $\mcE$. Observe that for $\mcE = \germe_X$, the sheaf $\CTF^1(\germe_X)$ is simply the cotangent sheaf $\CTF^1$. An $\F$-\emph{partial connection} (or simply a \emph{partial connection}, when the foliation $\F$ is clear in the context) on $\mcE$ is a $\C$-morphism
\begin{equation*}
    \begin{split}
        \nabla: \mcE & \longrightarrow \CTF^1(\mcE) \\
        s & \mapsto (v \mapsto \nabla_v(s))
    \end{split}
\end{equation*}
satisfying the Leibniz rule:
\begin{equation*}
    \nabla_v(f \cdot s) = v(f) \cdot s + f \cdot \nabla_v(s), \forall f \in \germe_X, s \in \mcE, v \in T_{\F}.
\end{equation*}
Observe that if $\F$ is the foliation by one leaf, that is, $T_{\F} =T_X$, a $\F$-partial connection is the same as a \emph{connection}. Most of the concepts presented in the following sections are straightforward generalizations of the corresponding concepts for connections (we refer to \cite{atiyah-57-zbMATH03128044,Deligne-70-zbMATH03385791,Katz-70-zbMATH03350023} as classical references to the theory of connections).

\begin{rmk}\label{Rmk: definition of partial connection}
    It is common to find in the literature a definition of a partial connection as a $\C$-morphism whose target is $\CTF^1 \otimes_{\germe_X} \mcE$ instead of $\CTF^1(\mcE)$ (see \cite[Definition 2.1]{baum-bott-72-zbMATH03423310}), and in general these definitions are not equivalent: although a partial connection with target $\CTF^1 \otimes_{\germe_X} \mcE$ induces a partial connection with target $\CTF^1(\mcE)$ by composition with the natural morphism $\CTF^1 \otimes_{\germe_X} \mcE \rightarrow \CTF^1(\mcE)$, there exist examples of partial connections with target $\CTF^1(\mcE)$ that cannot be described using $\CTF^1 \otimes_{\germe_X} \mcE$ (see Example \ref{Ex: partial connection on the singular ideal} below). Nevertheless, when the foliation is smooth, $\CTF^1$ is locally free and thus the natural morphism $\CTF^1\otimes_{\germe_X} \mcE \rightarrow \CTF^1(\mcE)$ is an isomorphism; therefore, in this case, both definitions coincide.
\end{rmk}

\begin{ex}(Derivation) The structure sheaf $\germe_X$ always admits the natural partial connection given by the \emph{derivation along the leaves of $\F$}, that is,
\begin{equation*}
    \begin{split}
        \df: \germe_X & \longrightarrow \CTF^1 \\
        f & \mapsto (v \mapsto v(f)).
    \end{split}
\end{equation*}
We say that a function $f\in \germe_X$ is a \emph{first integral of $\F$} if $\df(f)=0$, that is, for every $v\in T_{\F}$, $v(f)=0$. The set of first integrals of $\F$ forms a sheaf of rings. In the analytic category, it is non-zero and we denote it by $\germe_{X/\F}$. On the smooth locus of a foliation, it is simple to describe $\germe_{X/\F}$ locally. Let $x\in X$ be a smooth point of the foliation, and let $(x_1,\ldots,x_n)$ be a foliated system of coordinates defined on a neighborhood of $x$, such that $T_{\F}$ is generated by $\{\de / \de x_{q+1},\ldots, \de / \de x_n\}$. Then,
\[
\df(f) = 0 \iff \frac{\de f}{\de x_i} =0, q+1\le i \le n \iff f(x_1,\ldots,x_n) = f(x_1,\ldots,x_q), 
\]
and therefore $\germe_{X/\F,x} = \C\{x_1,\ldots,x_q \} \subset \germe_{X,x}$.
\end{ex}

\begin{ex}($\F$-invariant subvarieties) We say that a subvariety $Y\subset X$ is \emph{$\F$-invariant} if the ideal sheaf $I_Y$ is invariant under derivations in the tangent sheaf of $\F$, that is, for every $f \in I_Y$ and every $v\in T_{\F}$, we have $v(f) \in I_Y$. In terms of partial connections, this is the same as saying that the derivation along the leaves of $\F$, $\df: \germe_X \rightarrow \CTF^1$, induces a flat partial connection on the ideal sheaf $I_Y$. That is, $Y$ is $\F$-invariant if and only if we have the partial connection
\begin{equation*}
    \begin{split}
        \df: I_Y & \rightarrow \CTF^1(I_Y) \\
        f & \mapsto (v \mapsto v(f)).
    \end{split}
\end{equation*}
In particular, since $I_Y$ is generally not a locally free $\germe_X$-module, this provides an example of a flat partial connection on a coherent sheaf that is not locally free. This example contrasts with the well-known fact that a coherent sheaf with a flat connection must be locally free (see \cite[Proposition 8.8]{Katz-70-zbMATH03350023}).
\end{ex}

\begin{ex}\label{Ex: partial connection on the singular ideal}
    Let $\F$ be the foliation determined by the level sets of $f(x,y,z) = x^2+y^2+z^2$ on the complex space $X = \C^3$. It is easy to show that the ideal $\ideal{x,y,z}$ is $\F$-invariant, and thus it induces a partial connection
    \[
    \df: \ideal{x,y,z} \rightarrow \CTF^1(\ideal{x,y,z}).
    \]
    This connection does not arise from any partial connection of the type $\nabla: \ideal{x,y,z} \rightarrow \CTF^1 \otimes \ideal{x,y,z}$. Indeed, for $v\in T_{\F}$, let $\iota_v$ be the composition of the canonical map $\Omega_{\F}^1\otimes \ideal{x,y,z} \rightarrow \CTF^1(\ideal{x,y,z})$ with the evaluation map $\CTF^1(\ideal{x,y,z})\rightarrow \ideal{x,y,z}$. We have that $\CTF^1 = \Omega_X^1/\germe_X \cdot (xdx + ydy + zdz)$ and thus
    \[
    \eta(v) \in \ideal{x,y,z}, \forall \eta \in \CTF^1, v\in T_{\F}. 
    \]
    Hence, for every section $s \in \CTF^1 \otimes \ideal{x,y,z}$,  $\iota_v(s) \in \ideal{x,y,z}^2$. However, for $v=x\de/\de y - y \de/\de x \in T_{\F}$, we have $v(x) = -y \notin \ideal{x,y,z}^2$.
\end{ex}

\begin{ex}\label{Ex: bott connection}(Bott Connection) 
For every $v\in T_{\F}$, since $T_{\F}$ is involutive, the Lie derivative $L_v: T_X \rightarrow T_X$ leaves the subsheaf $T_{\F}$ invariant. Thus, it induces a morphism on the quotient, $L_v: T_X/T_{\F} \rightarrow T_X/T_{\F}$. Let $\pi: T_X \rightarrow T_X/T_{\F}$ be the natural quotient. We define the \emph{Bott connection} on $T_X/T_{\F}$ by the $\C$-morphism
\begin{equation*}
    \begin{split}
        \nabla_B: \frac{T_X}{T_{\F}} & \rightarrow \CTF^1\left(\frac{T_X}{T_{\F}} \right) \\
        \pi( v') & \mapsto \left(v \mapsto \pi ([v,v']) \right).
    \end{split}
\end{equation*}
The Bott connection was first defined in \cite{bott-67-zbMATH03286795} for holomorphic vector fields, and later generalized for more general foliations in \cite{baum-bott-72-zbMATH03423310,baum-bott-70-zbMATH03308090}. Throughout this work, the Bott connection will be used on several occasions to study the existence of transverse structures for a foliation.
\end{ex}

\subsection{Flat partial connections} Let $(\mcE,\nabla)$ be a partial connection on a coherent sheaf $\mc E$. We say that $\nabla$ is \emph{flat} if
\[
\nabla_{[v,w]} = \nabla_v \circ \nabla_w - \nabla_w \circ \nabla_v, \forall v,w\in T_{\F}.
\]
It is easy to verify that all examples of partial connections presented in Section \ref{Subsection: definition} are flat. We say that a section $s\in \mcE$ is \emph{flat} if $\nabla(s)=0$.  We denote the sheaf of flat sections by $\mcE^\nabla:=\{s\in\mcE:\nabla(s)=0\}$. 

\begin{prop}\label{P: flat connection and existence of flat basis}
    Let $\F$ be a smooth foliation on a complex manifold $X$, and let $(\mcE, \nabla)$ be a partial connection on a rank $r$ locally free sheaf. Then, $\nabla$ is flat if and only if for every $x\in X$, there exists a neighborhood $U\subset X$ of $x$ such that $\restr{\mc E}{U}$ is free and admits a basis of flat sections.
\end{prop}

In the following discussion and in the proof of Proposition \ref{P: flat connection and existence of flat basis}, we assume that $\F$ is smooth. We will explain how to interpret partial connections on locally free sheaves as systems of differential equations.

Let $(\mc E,\nabla)$ be a partial connection on a rank $r$ locally free sheaf.  Observe that, in this situation, the natural morphism $\CTF^1 \otimes \mcE \rightarrow \CTF^1(\mcE)$ is an isomorphism, and we could consider partial connections as $\C$-morphisms whose target is $\CTF^1 \otimes \mcE$ instead of $\CTF^1(\mcE)$ (see Remark \ref{Rmk: definition of partial connection}). Let $U\subset X$ be an open subset where $\restr{\mcE}{U}$ is free, and let us choose a basis of sections $\{e_1,\ldots, e_r\}$. There exists a collection of foliated 1-forms $\{\omega_{ij} \in \restr{\CTF^1}{U}\}$ such that $\nabla(e_i) = \sum_{j=1}^r \omega_{ji} \otimes e_j$, and using the Leibniz rule we calculate that
\[
\nabla\left(\sum_{i=1}^r f_i \cdot e_i \right) = \sum_{i=1}^r \df(f_i) \otimes e_i + \sum_{i,j=1}^r f_i \cdot \omega_{ji} \otimes e_j, \forall f_1,\ldots,f_r\in \germe_U. 
\]

Let $d$ be the dimension of $\F$, and $q$ be the codimension. Shrinking $U$ if necessary, let $(x_1,\ldots,x_q,y_1,\ldots,y_d)$ be a foliated system of coordinates, such that $T_{\F}$ is generated by $\{\de / \de y_1,\ldots, \de /\de y_d \}$. Let us abuse notation, and denote by $\{dy_1,\ldots, dy_d\}$ the dual basis of $\CTF^1$ on $U$. For every $1\le i,j \le r$, we write $\omega_{ij} =  -\sum_{k=1}^d A_{ijk}\cdot dy_k$ with respect to this basis.  Let $A_k = (A_{ijk})_{1\le i,j \le r}$ be an $r\times r$ matrix of functions.

Let us use the basis of $\CTF^1$ given above to study the flat sections of $\nabla$. Writing down the expression for $\nabla$ in this basis, we verify that a section $s=\sum_{i=1}^r f_i \cdot e_i$ is flat if and only if the collection of functions $(f_1,\ldots,f_r)$ satisfies the system of linear differential equations

\begin{equation*}
    \frac{\de }{\de y_k}\left(\begin{array}{c} f_1 \\ \vdots \\ f_r \end{array} \right) = A_k \cdot\left(\begin{array}{c} f_1 \\ \vdots \\ f_r \end{array} \right),  1\le k \le d.
\end{equation*}
Moreover, one can verify that the connection $\nabla$ is flat if and only if 
\begin{equation*}
    \frac{\de A_i}{\de y_j} - \frac{\de A_j}{\de y_i} = A_j \cdot A_i - A_i \cdot A_j, 1\le i,j\le d.
\end{equation*}

\begin{lemma}\label{L: integrable linear pfaff system}
    With the notation above, suppose that $\nabla$ is flat. Then, for every $r$-tuple of holomorphic functions $g_i(x_1,\ldots,x_q)$, $1\le i\le r$, defined near a point of $\{y=0\}$, the system of differential equations
    \begin{equation}\label{Eq: PVI}
        \begin{cases}
            \frac{\de }{\de y_k}\left(\begin{array}{c} f_1 \\ \vdots \\ f_r \end{array} \right) = A_k \cdot\left(\begin{array}{c} f_1 \\ \vdots \\ f_r \end{array} \right), \text{ for } 1\le k \le d, \\
            f_i(x_1,\ldots,x_q,0,\ldots,0) = g_i(x_1,\ldots,x_q), \text{ for } 1\le i \le r,
        \end{cases}
    \end{equation}
    admits a unique germ of a holomorphic solution near that point.
\end{lemma}

\begin{proof}
After shrinking $U$ around the chosen point $p$, let $Y=U\times\C^r$ with coordinates $(x_1,\ldots,x_q,y_1,\ldots,y_d,z_1,\ldots,z_r)$. On $Y$, consider the distribution $\mathcal D$ generated by the holomorphic vector fields
\[
D_k=\frac{\de}{\de y_k}+\sum_{i,j=1}^r A_{ijk}z_j\frac{\de}{\de z_i},\qquad 1\le k\le d.
\]
The flatness identity above implies that $\mathcal D$ is involutive and hence defines a holomorphic foliation on $Y$. Consider the submanifold
\[
Z=\{y_1=\cdots=y_d=0,\ z_i=g_i(x_1,\ldots,x_q),\ 1\le i\le r\},
\]
which is transverse to $\mathcal D$. By the holomorphic Frobenius theorem, a unique germ of a leaf of $\mathcal D$ passes through each point of $Z$. The union of these leaves defines a germ of a submanifold $L$ of dimension $q+d$. By construction, $L$ is invariant under the foliation and contains $Z$. By the implicit function theorem, $L$ is the graph of holomorphic functions $f_1,\ldots,f_r$ near the chosen point, that is,
\[
L=\{z_i=f_i(x_1,\ldots,x_q,y_1,\ldots,y_d),\ 1\le i\le r\}.
\]
The invariance of $L$ under $\mathcal D$ gives $\de/\de y_k(f_1,\ldots,f_r)^t=A_k(f_1,\ldots,f_r)^t$, and $Z\subset L$ gives $f(x_1,\ldots,x_q,0,\ldots,0)=g(x_1,\ldots,x_q)$. Uniqueness follows from the uniqueness of the leaf through each point of $Z$.
\end{proof}

We use Lemma \ref{L: integrable linear pfaff system} to prove Proposition \ref{P: flat connection and existence of flat basis}.
\begin{proof}[Proof of Proposition \ref{P: flat connection and existence of flat basis}]
    First, observe that the existence of a flat basis on any open subset $U\subset X$ implies that the connection is flat. Indeed, if $\restr{\mc E}{U}$ is free and admits a flat basis $\{e_1,\ldots,e_r\}$, we have that
    \[
    \nabla_{[v,w]}(e_i) = \nabla_v \circ \nabla_w(e_i) - \nabla_w \circ \nabla_v(e_i),  1\le i \le r, \forall v,w\in T_{\F},
    \]
    because both sides of the equation are zero. Therefore, $\nabla$ is flat. 

    Conversely, let us suppose now that $\nabla$ is flat. For every point $x\in X$, we consider an open neighborhood $U\subset X$ where we can keep the notation of Lemma \ref{L: integrable linear pfaff system}.  Let $(f_1^l,\ldots,f_r^l)$ be the solution of the System \eqref{Eq: PVI} for $g_1=0,\ldots,g_l=1,\ldots,g_r=0$, and let $s_l' = \sum_{i=1}^r f_i^l \cdot e_i$. By definition we must have $\nabla(s_l')=0$, and since the matrix $(f_i^l(x))_{1\le i,l\le r}$ is the identity, shrinking $U$ if necessary, $\{s_l'\}$ is still a basis for $\mcE$. Therefore, $\{s_l'\}$ forms a flat basis for $\mc E$ on a neighborhood of $x\in X$. This concludes the proof of the proposition. 
\end{proof}

\begin{cor}\label{C: correspondence}
    Let $\F$ be a smooth foliation on a complex manifold $X$. Then, for every flat partial connection $(\mcE,\nabla)$ on a locally free $\germe_X$-module of rank $r$, the sheaf of flat sections $\mcE^\nabla$ is a rank $r$ locally free $\germe_{X/\F}$-module. Conversely, for every locally free $\germe_{X/\F}$-module $\mb E$ of rank $r$, the sheaf $\germe_{X} \otimes_{\germe_{X/\F}} \mb E$ is a rank $r$ locally free $\germe_X$-module, and it is endowed with a unique flat partial connection $\nabla$ such that $\mcE^\nabla = \mb E$.
\end{cor}

\begin{ex}\label{Ex: partial connection induced by submersion} Let $\F$ be a foliation defined by a submersion $\phi:X\rightarrow Y$. Let $\mcE$ be a finite rank locally free sheaf of $\germe_Y$-modules. Since $\phi^{-1}\germe_Y \simeq \germe_{X/\F}$, it follows that $\phi^{-1}\mcE$ is endowed with a structure of locally free $\germe_{X/\F}$-modules. Then, the sheaf $\phi^*\mcE$ admits a flat partial connection $\nabla$ such that $(\phi^*\mcE)^\nabla = \phi^{-1}\mcE$.
\end{ex}

\subsection{The foliation induced by a flat partial connection } Still in the setting of flat partial connections on locally free sheaves, observe that $(\mcE, \nabla)$ induces a foliation on $Y=E(\mcE)$ (here and throughout this work, $E(\mcE) := \Spec( \Sym^{\bullet}(\mcE^*))$ is the total space of the locally free sheaf $\mcE$).  Sections of $\mcE^*$ correspond to fiberwise-linear functions on $Y$. Thus, for every $v\in T_{\F}$, the dual partial connection $\nabla_v^*$ on $\mcE^*$ induces a vector field $\Tilde{v}\in T_Y$ projecting to $v$. Concretely, let $\{e_1,\ldots,e_r\}$ be a local basis for $\mcE$, and $(y_1,\ldots,y_r)$ the corresponding system of coordinates on $Y$. If $\nabla_v(e_i) = \sum f_{ij}(x) \cdot e_j$, then $\Tilde{v}$ can be written as
\[
 \Tilde{v} = v - \sum_{i,j=1}^r f_{ij}(x) \cdot y_i \cdot \frac{\de }{\de y_j}.
\]
Since $\nabla$ is flat, so is $\nabla^*$, and therefore $[\Tilde{v},\Tilde{w}]=\widetilde{[v,w]}$. Thus, a flat partial connection $\nabla$ on $\mcE$ induces a foliation $\pi^*T_{\F} \rightarrow T_Y$ such that we have the commutative diagram
\begin{equation}\label{eq: diagram of the lifting}
\begin{tikzpicture}
    \matrix(m)[matrix of math nodes, column sep = 2em, row sep = 2em]
    {
    & T_Y \\
    \pi^*T_{\F} & \pi^*T_X \\ 
    };
    \path[->]
    (m-2-1) edge (m-1-2) edge (m-2-2)
    (m-1-2) edge node[right]{$d\pi$} (m-2-2)
    ;
\end{tikzpicture},
\end{equation}
where $\pi^*T_{\F} \rightarrow \pi^*T_X$ is the pullback of the inclusion $T_{\F} \subset T_X$.

\begin{ex}\label{Ex: foliation + partial connection induced by submersion}
Following the notation of Example \ref{Ex: partial connection induced by submersion}, the flat partial connection $\nabla$ induces a foliation $\G$ on $E(\phi^*\mcE)$. Using local coordinates, it is easy to see that $\G$ is the foliation induced by the bundle morphism (which is also a submersion) $E(\phi^*\mcE) \rightarrow E(\mcE)$.
\end{ex}

\subsection{The category of partial connections} Let $(\mcE,\nabla)$ and $(\mcE',\nabla')$ be partial connections, and let $\phi: \mcE \rightarrow \mcE'$ be an $\germe_X$-linear morphism. We say that $\phi$ is \emph{horizontal} (with respect to $\nabla$ and $\nabla'$) if the diagram
\begin{equation*}
    \begin{tikzpicture}
        \matrix(m)[matrix of math nodes, column sep = 2em, row sep = 2em]
        {
        \mcE & \mcE' \\
        \CTF^1(\mcE) & \CTF^1(\mcE') \\
        };
        \path[->]
        (m-1-1) edge node[above]{$\phi$} (m-1-2) edge node[left]{$\nabla$} (m-2-1)
        (m-2-1) edge node[above]{$\widetilde{\phi}$} (m-2-2)
        (m-1-2) edge node[right]{$\nabla'$} (m-2-2)
        ;
    \end{tikzpicture}
\end{equation*}
commutes, where $\widetilde{\phi}: \CTF^1(\mcE) \rightarrow \CTF^1(\mcE')$ is the natural morphism induced by $\phi$.

We define the \emph{category of (flat) partial connections} as the category where the objects are flat partial connections on coherent sheaves, and the morphisms between objects are horizontal morphisms. 

\begin{prop}
    The category of (flat) partial connections on coherent sheaves is abelian.
\end{prop}
The proof of the above proposition is a straightforward diagram chase. In the following paragraphs, let us describe two important constructions that also hold in the category of partial connections: the tensor product and the $\mcHom$ operator.

Let $(\mcE,\nabla)$ and $(\mcE',\nabla')$ be partial connections on coherent sheaves. We define the \emph{tensor product} $(\mcE,\nabla) \otimes (\mcE',\nabla')$ as the sheaf $\mcE \otimes_{\germe_X} \mcE'$ endowed with the partial connection $\nabla \otimes \nabla'$ defined as
\[
(\nabla\otimes \nabla')_v(s\otimes s') := \nabla_v(s) \otimes s' + s\otimes \nabla_v'(s'), \forall v\in T_{\F},s\in \mcE,s'\in \mcE'.
\]
Moreover, when both $\nabla,\nabla'$ are flat, one can directly verify that $\nabla \otimes \nabla'$ is also flat.

Let us now define the \emph{$\mcHom$ operator in the category of partial connections}. We define $\mcHom((\mcE,\nabla), (\mcE',\nabla'))$ as the sheaf $\mcHom_{\germe_X}(\mcE,\mcE')$ endowed with the partial connection $\nabla''$ defined by
\[
\nabla_v''(\phi)(s) = \nabla'_v(\phi(s)) - \phi(\nabla_v(s)), \forall v\in T_{\F}, \phi \in \mcHom_{\germe_X}(\mcE,\mcE'), s\in \mcE.
\]
As before, when both $\nabla$ and $\nabla'$ are flat, we verify that $\nabla''$ is also flat. In particular, when $(\mcE',\nabla')=(\germe_X,\df)$, we define $\mcHom((\mcE,\nabla), (\germe_X,\df)) = (\mcE^*,\nabla^*)$ as the \emph{dual} of the connection $(\mcE,\nabla)$. One can easily verify that when $\mcE$ is a reflexive sheaf, then there is a natural isomorphism between $(\mcE,\nabla)$ and $(\mcE^{**},\nabla^{**})$. In this sense, the partial connection $(\mcE,\nabla)$ is also reflexive in the category of partial connections.

\begin{ex}  Dualizing the Bott connection on $T_X/T_{\F}$ from Example \ref{Ex: bott connection} gives a partial connection on $\CNF=(T_X/T_{\F})^*$, which we also denote by $\nabla_B$. Let $v\in T_{\F}$, $v'\in T_X$, $\omega\in\CNF$, and $u=\pi(v')\in T_X/T_{\F}$. By the formula for the dual connection,
\[
 (\nabla_B)_v(\omega)(u)=v(\omega(u))-\omega((\nabla_B)_v(u))
=v(\omega(v'))-\omega([v,v']).
\]
 Recall that
\[
 d\omega(v,v')=v(\omega(v'))-v'(\omega(v))-\omega([v,v']).
\]
 Since $\omega(v)=0$, the right-hand side equals both $d\omega(v,v')$ and $\mathcal L_v(\omega)(v')$. Thus $(\nabla_B)_v(\omega)=\mathcal L_v(\omega)$ on $\CNF$. Dualizing once more gives the Bott connection on $N_{\F}=(\CNF)^*$.
\end{ex}

\subsection{Extensions of partial connections} 
Let $\nabla: \mcE \rightarrow \Omega_X^1(\mcE)$ be a connection. We define the \emph{restriction of $\nabla$} as the partial connection $\nabla_0$ defined by the following commutative diagram:
\[
    \begin{tikzpicture}
        \matrix(m)[matrix of math nodes, column sep = 2em, row sep = 2em]
        {
        \mcE & \Omega_X^1(\mcE) \\
        & \CTF^1(\mcE) \\
        };
        \path[->]
        (m-1-1) edge node[above]{$\nabla$}(m-1-2) edge node[shift={(-0.2,-0.2)}]{$\nabla_0$}(m-2-2)
        (m-1-2) edge node[right]{$\mrrestr$} (m-2-2)
        ;
    \end{tikzpicture},
\]
where $\mrrestr: \Omega_X^1(\mcE) \rightarrow \CTF^1(\mcE)$ is the restriction of 1-forms induced by the inclusion $T_{\F} \subset T_X$. Conversely, we say that $\nabla$ is the \emph{extension of} $\nabla_0$.

\begin{prop}\label{prop: transversely affine}
    Let $\F$ be a smooth codimension one foliation on $X$. Then, $\F$ is a transversely affine foliation if and only if the Bott connection on the conormal sheaf admits a flat extension.
\end{prop}
\begin{proof}

Let us first suppose that $\F$ admits a transverse affine structure $\mc A = \{f_i: U_i \rightarrow \C \}$. For every pair $(i,j)$ such that $U_i\cap U_j \neq \emptyset$, there exists $a_{ij} \in \C^*, b_{ij} \in \C$ such that $f_i = a_{ij} \cdot f_j + b_{ij}$ on $U_i \cap U_j$. Thus, $df_i = a_{ij} \cdot df_j$, that is, we have a local system $\mc S \subset \CNF$ locally generated by $df_i$. The local system $\mc S$ induces a flat connection $\hatnabla$ on $\CNF$, which extends the Bott connection, as one easily verifies.

Conversely, let $\hatnabla$ be a flat extension of the Bott connection, and choose local generators $\omega_i\in\CNF(U_i)$ of its sheaf of horizontal sections. These forms are closed. Indeed, for $v\in T_{\F}$ and $v'\in T_X$, set $u=\pi(v')$. The calculation in the preceding example gives
\[
d\omega_i(v,v')=(\nabla_B)_v(\omega_i)(u)=(\hatnabla)_v(\omega_i)(u)=0.
\]
Since $\F$ has codimension one, $d\omega_i$ also vanishes on two vectors transverse to $\F$, so $d\omega_i=0$. After shrinking the $U_i$, choose functions $f_i:U_i\to\C$ with $df_i=\omega_i$. On every connected component of $U_i\cap U_j$, horizontal generators differ by a constant $a_{ij}\in\C^*$; hence $df_i=a_{ij}\,df_j$, and integration gives $f_i=a_{ij}f_j+b_{ij}$ for some $b_{ij}\in\C$. Thus the $f_i$ define a transverse affine structure for $\F$.
\end{proof}

\begin{rmk}
    This is a generalization of the well-known fact that an affine structure for a curve $C$ is a connection on $T_C$ (see \cite[Lemma 1]{gunning-67-zbMATH03233043}). See also \cite[Section 2.2]{cousin-pereira-14-zbMATH06399599} for the same result relating singular transverse affine structures and flat meromorphic extensions of the Bott connection. 
\end{rmk}

\section{Jets of flat partial connections} \label{Section: jets}

Throughout this section, $\F$ denotes a smooth foliation of codimension $q$ on a complex manifold $X$, and our goal is to describe the construction of \emph{jets of flat partial connections on locally free sheaves}. That is, starting with a flat partial connection $(\mcE,\nabla)$ on a locally free sheaf of $\germe_X$-modules, we define, for each $k\ge 0$, the \emph{$k$-th sheaf of transverse jets of $(\mcE,\nabla)$} as a locally free sheaf $\jetXF^k(\nabla) \subset \jetX^k(\mcE)$, endowed with a natural flat partial connection $\nabla^k$, such that the jets of the flat sections of $\nabla$ are flat sections of $\nabla^k$.

\subsection{Jets}  Before starting the construction of \emph{jets of flat partial connections}, let us recall the main definitions of the theory of jets and set some notation. The references for this section are \cite[Chapter~2]{berthelot-78-zbMATH03595321} and \cite[Section~16]{ega-VI-zbMATH03245973}.

\subsubsection*{The ring of jets} Let $X$ be a complex manifold. Let $I\subset \germe_{X} \otimes_{\C} \germe_X$ be the kernel of the morphism of sheaves of rings $\germe_X \otimes_{\C} \germe_X \rightarrow \germe_X$ given by $f\otimes g \mapsto f \cdot g$. It is easy to see that $I$ is the ideal sheaf generated by elements of the form $f\otimes g - g\otimes f, \forall f,g\in \germe_X$. For every $k\ge 0$, we define the \emph{ring of the $k$-jets} over $X$ by
\[
\jetX^k := \frac{\germe_X \otimes_{\C} \germe_X}{I^{k+1}}.
\]
We abuse notation and denote by $f\otimes g \in \jetX^k$ the image of $f\otimes g \in \germe_X \otimes_{\C} \germe_X$ under the natural projection $\germe_X \otimes_{\C} \germe_X\rightarrow \jetX^k$. 

\begin{rmk}\label{Rmk: different definitions of jets} This definition of the sheaf of $k$-jets can be found in \cite[Chapter 2]{berthelot-78-zbMATH03595321}, and, as explained in \cite[Section 16.3.7]{ega-VI-zbMATH03245973}, it coincides with the definition of the sheaf of principal parts given in \cite[Definition 16.3.1]{ega-VI-zbMATH03245973}. Additionally, there is the definition of the \emph{jet bundle}, which is more commonly encountered in the context of differential geometry (compare \cite[Definition 6.2.3]{saunders-89-zbMATH00042050} with \eqref{Eq: k-jet of function in coordinates} below).
\end{rmk}

\begin{rmk}
In the context of algebraic geometry, it is usual to denote the sheaf of jets (which, as explained, coincides with the sheaf of principal parts) by $\mc P^k_X$ rather than $\jetX^k$, as is the case in both \cite{berthelot-78-zbMATH03595321} and \cite{ega-VI-zbMATH03245973}. However, we have chosen to retain the notation $\jetX^k$, which is more common in the context of differential geometry, as it seems more closely aligned with the applications we have in mind. 
\end{rmk}

Observe that $\jetX^k$ inherits  the structure of a sheaf of rings from $\germe_X \otimes_{\C} \germe_X$. Moreover, from the definition, $\jetX^k$ admits two structures of $\germe_X$-modules: the \emph{left structure} (respectively, the \emph{right structure}) is the $\germe_X$-module structure induced by the morphism of sheaves of rings $\germe_X \rightarrow \jetX^k$ given by $f \mapsto f\otimes 1$ (respectively, $f\mapsto 1\otimes f)$. We take the left structure as the \emph{canonical one}, and for that reason we abuse notation and denote the element $f\otimes 1 \in \jetX^k$ simply by $f\in \jetX^k$. 

For the right structure, we denote the morphism $f\mapsto 1\otimes f $ by $d^k:\germe_X \rightarrow \jetX^k$. For every function $f\in \germe_X$, we refer to $d^k(f)$ as the \emph{$k$-jet} of $f\in \germe_X$.  This name is justified since in coordinates $d^k(f)$ represents the $k$-jet of the function $f$, as defined in differential geometry (see \eqref{Eq: k-jet of function in coordinates} below for the calculation of $d^k(f)$ in coordinates).

Observe that there is a natural short exact sequence associated to the sheaf of jets. Indeed, for every $k\ge 1$, we  consider the natural short exact sequence
\begin{equation*}
    0 \rightarrow \frac{I^k}{I^{k+1}} \rightarrow \frac{\germe_X \otimes_{\C} \germe_X}{I^{k+1}} \rightarrow \frac{\germe_X \otimes_{\C} \germe_X}{I^{k}} \rightarrow 0.
\end{equation*}
Recall that we have the isomorphism $I/I^2 \simeq \Omega_X^1$ given by $1\otimes f-  f\otimes 1 \mapsto df$. This isomorphism induces, for every $k\ge 1$, the isomorphism $I^k/I^{k+1}\simeq \Sym^k(\Omega_X^1)$ given by
\[
(1\otimes f_1 - f_1\otimes 1) \cdots (1\otimes f_k -f_k \otimes 1) \mapsto df_1 \cdots df_k.
\]
Hence, we deduce the short exact sequence 
\begin{equation}\label{Eq: short exact sequence of rings of jets}
    0 \rightarrow \Sym^k(\Omega_X^1) \rightarrow \jetX^k \rightarrow \jetX^{k-1} \rightarrow 0,
\end{equation}
where the morphism $\Sym^k(\Omega_X^1) \rightarrow \jetX^k$ is given by
\[
df_1 \cdots df_k \mapsto (d^k(f_1) - f_1)\cdots (d^k(f_k) - f_k).
\]

Let $(x_1,\ldots,x_n)$ be a system of coordinates for the manifold $X$ on an open subset $U\subset X$. As shown in \cite[Equations 16.11.1.5 and 16.11.1.6]{ega-VI-zbMATH03245973}, one can construct two natural bases for $\jetX^k$:
\begin{equation}\label{Eq: basis for the ring of jets}
    \begin{split}
        \mc B_{1,X}^k &  = \left\{d^k(\mbf x^{\mbf i}); |\mbf i| \le k \right\}, \text{ and } \\
        \mc B_{2,X}^k &  = \left\{\mbf \zeta^{\mbf i}; |\mbf i| \le k \right\},
    \end{split}
\end{equation}
where $\mbf x= (x_1,\ldots,x_n)$, $\mbf i = (i_1,\ldots,i_n)$, $|\mbf i | = i_1+ \cdots +i_n$, $ \mbf x^{\mbf i} = x_1^{i_1} \cdots x_n^{i_n}$ and $\mbf \zeta^{\mbf i} = (d^k(x_1)-x_1)^{i_1} \cdots (d^k(x_n)-x_n)^{i_n}$. The elements of the basis $\mc B^k_{1,X}$ and $\mc B^{k}_{2,X}$ are related by the following formulas:
\begin{equation}\label{Eq: change of the bases for the ring of jets}
    \begin{split}
        d^k(\mbf x^{\mbf i}) & = \sum_{\mbf j \le \mbf i}\binom{\mbf i}{\mbf j} \mbf x^{\mbf i - \mbf j}\cdot \mbf \zeta^{\mbf j}, \text{ and } \\
        \mbf \zeta^{\mbf i}  & = \sum_{\mbf j \le \mbf i} (-1)^{|\mbf i-\mbf j|} \binom{\mbf i}{\mbf j} \mbf x^{\mbf i - \mbf j}\cdot d^k(\mbf x^{\mbf j}).
    \end{split}
\end{equation}
Observe that the basis $\mc B_{2,X}^k$ is the natural basis used to construct jets from the differential-geometric perspective. Indeed, for every $f\in \germe_U$, we have 
\begin{equation}\label{Eq: k-jet of function in coordinates}
    d^k(f) = \sum_{|\mbf i| \le k} \frac{1}{\mbf i !}\frac{\de^{ |\mbf i|} f}{\de \mbf x^{ \mbf i}}\cdot \zeta^{\mbf i},     
\end{equation}
that is, the coefficients of $d^k(f)$ with respect to the basis $\mc B_{2,X}^k$ are the derivatives of $f$ up to order $k$.

\subsubsection*{Jets of sections of a sheaf} Let $\mcE$ be a sheaf of $\germe_X$-modules. For every $k\ge 0$, we define the \emph{sheaf of the $k$-jets of sections of $\mcE$} by
\[
\jetX^k(\mcE) := \jetX^k \otimes \mcE,
\]
where $\otimes$ stands for the tensor product of $\mcE$ and $\jetX^k$ with respect to the right $\germe_X$-algebra structure. By definition, $\jetX^k(\mcE)$ is naturally endowed with two $\germe_X$-module structures: the \emph{left (or canonical) structure} is defined by the product
\[
 f \cdot (a\otimes s) = (fa) \otimes s, \forall f\in \germe_X, a\in \jetX^k, s \in \mcE,
\]
and the \emph{right structure} is defined by the product
\[
(a\otimes s) \cdot f = a\otimes (fs), \forall f\in \germe_X, a\in \jetX^k,s\in \mcE.
\]

Since the $k$-jet morphism $d^k: \germe_X \rightarrow \jetX^k$ is $\germe_X$-linear with respect to the right structure of $\jetX^k$, the tensor product with $\mcE$ induces the morphism
\begin{equation*}
    \begin{split}
        d^k_{\mc E}: \mcE & \longrightarrow \jetX^k(\mcE) \\
        s & \mapsto 1\otimes s,
    \end{split}
\end{equation*}
which is still $\germe_X$-linear with respect to the right $\germe_X$-module structure of $\jetX^k(\mcE)$. When the sheaf $\mcE$ is clear in the context, we omit it and denote $d^k_{\mcE}$ by $d^k$. Moreover, for every $s\in \mcE$, we say that $d^k(s)$ is the \emph{$k$-jet of $s$}.

As in the case of the ring of jets, there is a natural short exact sequence associated to the definition of sheaves of jets. Indeed, applying the tensor product with $\mcE$ to the short exact sequence (\ref{Eq: short exact sequence of rings of jets}) we have the short exact sequence
\begin{equation}\label{Eq: short exact sequence of the sheaf of jets of sections}
    0 \rightarrow \Sym^k(\Omega_X^1)(\mcE) \xrightarrow{\iota} \jetX^k(\mcE) \xrightarrow{\pi} \jetX^{k-1}(\mcE) \rightarrow 0.
\end{equation}

Since $\jetX^k$ admits a structure of a sheaf of rings, $\jetX^k(\mcE)$ is also naturally endowed with a $\jetX^k$-module structure. When $\mcE$ is a locally free sheaf of $\germe_X$-modules, then $\jetX^k(\mcE)$ is a locally free sheaf of $\jetX^k$-modules, and thus $\jetX^k(\mcE)$ is also a locally free sheaf of $\germe_X$-modules, with both left and right structures. Let us explicitly describe a basis for $\jetX^k(\mcE)$. Let $(x_1,\ldots,x_n)$ be a system of coordinates for the manifold $X$, and $\{e_1,\ldots,e_r\}$ be a basis for $\mcE$, both on an open subset $U\subset X$. Using the basis for $\jetX^k$ we presented in \eqref{Eq: basis for the ring of jets}, we have the two bases for $\jetX^k(\mcE)$ on $U$:
\begin{equation}\label{Eq: basis for the sheaf of jets of a sheaf}
    \begin{split}
        \mc B_{1,\mcE}^k &  = \left\{d^k(\mbf x^{\mbf i} \cdot e_j); |\mbf i| \le k, 1\le j \le r \right\}, \text{ and } \\
        \mc B_{2,\mcE}^k &  = \left\{\mbf \zeta^{\mbf i} \cdot d^k(e_j); |\mbf i| \le k, 1\le j \le r \right\}.
    \end{split}
\end{equation}
These bases are also related by a change of coordinates similar to (\ref{Eq: basis for the ring of jets}). Finally, for every $s \in \mcE$, writing $s = \sum_{i=1}^r f_i \cdot e_i$, we have that
\begin{equation}\label{Eq: k-jet of sections in coordinates}
    d^k(s) = \sum_{|\mbf i |\le k, 1\le j\le r} \frac{1}{\mbf i!}\frac{\de^{|\mbf i|} f_j}{\de \mbf x^{\mbf i}} \cdot \zeta^{\mbf i} \cdot d^k(e_j),
\end{equation}
and thus $d^k(s)$ corresponds to the usual notion of $k$-jet of $s$ as described in \cite[Definitions 6.2.2-6.2.4]{saunders-89-zbMATH00042050}.

\subsubsection*{Connections and the sheaf of jets}
As we described in  the sequence (\ref{Eq: short exact sequence of the sheaf of jets of sections}), for every $\germe_X$-module $\mcE$, we have a natural short exact sequence associated to $\jetX^k(\mcE)$. In particular, we obtain the short exact sequence associated to the sheaf of the first jets of $\mcE$:
\begin{equation}\label{Eq: short exact sequence of the sheaf of first jets of sections}
    0 \rightarrow \Omega_X^1(\mcE) \xrightarrow{\iota} \jetX^1(\mcE) \xrightarrow{\pi} \mcE \rightarrow 0.
\end{equation}
The following standard correspondence between connections on $\mcE$ and splittings of \eqref{Eq: short exact sequence of the sheaf of first jets of sections} follows from \cite[Theorem 5]{atiyah-57-zbMATH03128044}. See also \cite[Proposition 2.9]{berthelot-78-zbMATH03595321}.

\begin{prop}\label{P: meromorphic connections, bijection} Let $X$ be a complex manifold, and let $\mcE$ be a coherent $\germe_X$-module. Then, there exists a natural bijection between:
\begin{enumerate}[label = (\roman*)]
    \item  connections $\nabla: \mcE \rightarrow \Omega_X^1(\mcE)$; and
    \item  splittings of the short exact sequence (\ref{Eq: short exact sequence of the sheaf of first jets of sections}).
\end{enumerate}
\end{prop}

\subsection{Transverse jets}\label{Subsection: transverse jets} Let $(\mcE ,\nabla)$ be a flat partial connection on a locally free sheaf. We define the \emph{sheaf of $k$-jets of flat sections of $\nabla$} by
\[
\jetXF^k(\mcE^\nabla) := \left\{\sum f_i \cdot d^k(s_i), f_i\in \germe_{X/\F}, s_i \in \mcE^\nabla \right\} \subset \jetX^k(\mcE), 
\]
and the \emph{$k$-th sheaf of transverse jets of $(\mcE,\nabla)$} by
\[
\jetXF^k(\nabla) := \left\{\sum f_i \cdot d^k(s_i), f_i\in \germe_{X}, s_i \in \mcE^\nabla \right\} \subset \jetX^k(\mcE).
\]

\begin{prop}\label{P: jets are locally free sheaves}
    Let $\F$ be a smooth foliation on a complex manifold $X$, and let $(\mcE,\nabla)$ be a flat partial connection on a locally free sheaf. Then, $\jetXF^k(\nabla)$ (respectively, $\jetXF^k(\mcE^\nabla)$) is a locally free sheaf of $\germe_X$-modules (respectively, $\germe_{X/\F}$-modules).
\end{prop}

\begin{proof}
    Let $(x_1,\ldots,x_q,y_1,\ldots,y_{n-q})$ be a foliated system of coordinates on an open subset $U\subset X$ where $\F$ is defined by $dx_1,\ldots,dx_q$. Shrinking $U$ if necessary, let $\{e_1,\ldots,e_r\}$ be a flat basis of $\restr{\mcE}{U}$. Let $\mbf x = (x_1,\ldots,x_q)$ and $\mbf i = (i_1,\ldots,i_q)$. Let us first verify that 
    \[
        \mc B^k_{1, \nabla} := \{ d^k(\mbf x^{ \mbf i} \cdot e_j); |\mbf i|\le k, 1\le j \le r\}
    \]
    is a basis for $\jetXF^k(\nabla)$ as a sheaf of $\germe_X$-modules. First, since $\mbf x^{ \mbf i} \cdot e_j \in \mcE^\nabla$, we have $d^k(\mbf x^{ \mbf i} \cdot e_j) \in \jetXF^k(\nabla)$. Therefore,
    \[
    \bigoplus_{|\mbf i|\le k, 1\le j \le r} \germe_X \cdot d^k(\mbf x^{ \mbf i} \cdot e_j) \subset  \restr{\jetXF^k(\nabla)}{U} \subset \restr{\jetX^k(\mcE)}{U}.
    \]
    For the reverse inclusion, observe that every $s\in \mcE^\nabla$ can be written uniquely as $s = \sum_{j=1}^r f_j \cdot e_j$ with $f_j\in \germe_{X/\F}, 1\le j \le r$. Using the description of $d^k: \mcE \rightarrow \jetX^k(\mcE)$ given by \eqref{Eq: k-jet of sections in coordinates}, since $f \in \germe_{X/\F}$ (which is the same as saying that $\de f/\de y_i =0$ for $1\le i \le n-q$), it follows that 
    \[
    d^k(s) = \sum_{|\mbf i |\le k, 1\le j\le r} \frac{1}{\mbf i!}\frac{\de^{|\mbf i|} f_j}{\de \mbf x^{\mbf i}} \cdot \zeta^{\mbf i} \cdot d^k(e_j) \in \bigoplus_{|\mbf i|\le k, 1\le j \le r} \germe_X \cdot \zeta^{\mbf i} \cdot d^k(e_j),
    \]
    where $\mbf \zeta^{\mbf i} = (d^k(x_1)-x_1)^{i_1} \cdots (d^k(x_q)-x_q)^{i_q}$. Finally, using the change of bases between $\mc B^k_{1,\mcE}$ and $\mc B^k_{2,\mc E}$ explicitly described in \eqref{Eq: change of the bases for the ring of jets}, it follows that every $\zeta^{\mbf i}\cdot d^k(e_j)$ can be written in terms of the basis $\mc B^k_{1,\nabla}$. Thus, 
    \[
    \restr{\jetXF^k(\nabla)}{U} \subset \bigoplus_{|\mbf i|\le k, 1\le j \le r} \germe_X \cdot d^k(\mbf x^{ \mbf i} \cdot e_j),
    \]
and we conclude that $\mc B^k_{1,\nabla}$ is a basis for $\jetXF^k(\nabla)$ on the open subset $U\subset X$. Therefore, $\jetXF^k(\nabla)$ is a locally free sheaf.

    With the same reasoning, we prove that $\mc B^k_{1,\nabla}$ is a basis for $\restr{\jetXF^k(\mcE^\nabla)}{U}$ as a sheaf of $\germe_{X/\F}$-modules. Therefore, $\jetXF^k(\mcE^\nabla)$ is a locally free sheaf of $\germe_{X/\F}$-modules.
\end{proof}

From the proof of the proposition above, we deduce that
\[
\mc B^k_{2,\nabla} := \left\{\mbf \zeta^{\mbf i}\cdot d^k(e_j); |\mbf i| \le k, 1\le j \le r \right\}
\]
    is also a basis for both $\jetXF^k(\mcE^\nabla)$ and $\jetXF^k(\nabla)$ on the open subset $U\subset X$.

\begin{cor}
    Let $\F$ be a smooth foliation on a complex manifold $X$, and let $(\mcE,\nabla)$ be a flat partial connection on a locally free sheaf. Then, there is a natural isomorphism 
    \begin{equation*}
        \begin{split}
            \germe_X \otimes_{\germe_{X/\F}} \jetXF^k(\mcE^\nabla) & \longrightarrow \jetXF^k(\nabla) \\
            f\otimes a & \mapsto f\cdot a.
        \end{split}
    \end{equation*}
    Moreover, there exists a unique flat partial connection $\nabla^k$ on $\jetXF^k(\nabla)$ such that $(\jetXF^k(\nabla))^{\nabla^k} = \jetXF^k(\mcE^\nabla)$.
\end{cor}
\begin{proof} The morphism above sends $1\otimes b$ to $b$ for every $b\in \mathcal B^k_{1,\nabla}$, where $\mathcal B^k_{1,\nabla}$ is the local basis for both $ \jetXF^k(\mcE^\nabla) $ and $\jetXF^k(\nabla)$ described in Proposition \ref{P: jets are locally free sheaves}.  Hence it is an isomorphism. The second assertion follows by applying Corollary \ref{C: correspondence} to the locally free $\germe_{X/\F}$-module $\jetXF^k(\mcE^\nabla)$.
\end{proof}

We will refer to the pair $(\jetXF^k(\nabla),\nabla^k)$ as the \emph{$k$-jet of the flat partial connection $(\mcE,\nabla)$}. 

\begin{cor}\label{C: short exact sequence of jets of flat partial connections} Let $\F$ be a smooth foliation on a complex manifold $X$, and let $(\mcE,\nabla)$ be a flat partial connection on a locally free sheaf. Then, for every $ k\ge 1$, the short exact sequence associated to the sheaf of $k$-jets of sections of $\mcE$ given by  (\ref{Eq: short exact sequence of the sheaf of jets of sections}) induces a short exact sequence of flat partial connections 
\begin{equation*}
    0 \rightarrow (\Sym^{k}(\CNF) \otimes \mcE, \nabla_B \otimes \nabla) \xrightarrow{\iota} (\jetXF^k(\nabla),\nabla^k) \xrightarrow{\pi} (\jetXF^{k-1}(\nabla), \nabla^{k-1}) \rightarrow 0,
\end{equation*}
    where $\nabla_B$ here stands for the natural flat partial connection on $\Sym^{k}(\CNF)$ induced by the Bott connection. 
\end{cor}
\begin{proof} After choosing a foliated coordinate chart and a flat basis for $\mcE$, we can consider one of the local bases for the jets described in Proposition \ref{P: jets are locally free sheaves}. Using $\mathcal B^k_{2,\nabla}$ for example, we see that $\pi$ is surjective and that its kernel is identified with $\Sym^k(\CNF)\otimes\mcE$. Moreover, since these bases are flat for the respective connections, $\nabla^k$, $\nabla^{k-1}$, and $\nabla_B\otimes\nabla$ act by differentiating coefficients only along the leaves of the foliation. The maps $\iota$ and $\pi$ therefore commute with the partial connections, and the sequence is indeed short exact of flat partial connections.
\end{proof}

We refer to this exact sequence as the \emph{short exact sequence of the $k$-th sheaf of transverse jets}.

\begin{ex} For $(\mcE,\nabla)=(\germe_X,\df)$ and $k=1$, the exact sequence above has the canonical horizontal splitting induced by $f\mapsto f\cdot d^1(1)$. Consequently, $\jetXF^1(\df)\simeq\germe_X\oplus\CNF$, and under this identification $\nabla^1=\df\oplus\nabla_B$. In particular, the induced partial connection on the $\CNF$ summand is precisely the Bott connection.
\end{ex}

\section{Transverse homogeneous linear differential equations}\label{Section: transverse differential equations}

\subsection{Definition}
Let $\F$ be a smooth foliation of codimension one on a complex manifold $X$, and let $(\mcE, \nabla)$ be a flat partial connection on a locally free sheaf. For every $k\ge 1$, consider the short exact sequence of flat partial connections described by Corollary \ref{C: short exact sequence of jets of flat partial connections}. We define a \emph{system of transverse homogeneous linear differential equations of order $k$ on $(\mcE,\nabla)$} (or simply a transverse differential equation, when it is clear in the context)  as a horizontal splitting of the short exact sequence of the $k$-th sheaf of transverse jets, that is, a horizontal $\germe_X$-linear morphism
\[
\sigma: (\jetXF^k(\nabla),\nabla^k) \rightarrow (\CNF^{\otimes k} \otimes \mcE, \nabla_B \otimes \nabla)
\]
such that $\sigma \circ \iota:\CNF^{\otimes k} \otimes \mcE \rightarrow \CNF^{\otimes k} \otimes \mcE$ is the identity morphism. Consider the $k$-jet morphism $d^k: \mcE^\nabla \rightarrow \jetXF^k(\nabla)$, and let $E = \sigma \circ d^k: \mcE^\nabla \rightarrow \CNF^{\otimes k} \otimes \mcE$. We say that a section $s\in \mcE^\nabla$ is a \emph{solution of $\sigma$} (or a \emph{solution of $E$}) if $E(s)=0$. We write $\mcE^E:=\{s\in\mcE^\nabla:E(s)=0\}$ for the sheaf of solutions. We abuse notation and also call $E$ the transverse differential equation.

Let $x\in X$ and, on a neighborhood $U$ of $x$, let $(x_1,x_2,\ldots,x_n)$ be a system of coordinates centered at $x$, such that $\F$ is defined by $dx_1$. Shrinking $U$ if necessary, suppose $\mcE$ is free with flat basis $\{e_1,\ldots,e_r\}$. Let $\zeta = d^k(x_1) - x_1 \in \jetX^k$ and consider the basis $\mc B^k_{2,\nabla} = \{\zeta^i\cdot d^k(e_j), 0\le i \le k, 1\le j \le r\}$ of $\restr{\jetXF^k(\nabla)}{U}$. Applying $\sigma$ to the elements of this basis, there exist holomorphic functions $a_{ijl}\in\germe_{\C,0}$ such that
\begin{equation}\label{Eq: transverse differential equation in coordinates}
    \begin{split}
        \sigma\left(\frac{\zeta^i}{i!}\cdot d^k(e_j) \right) & = \sum_{l=1}^r a_{ijl}(x_1) \cdot \frac{dx_1^k}{k!} \otimes e_l \in \CNF^{\otimes k} \otimes \mc E, 0\le i \le k-1, 1\le j \le r, \\
        \sigma\left(\frac{\zeta^k}{k!}\cdot d^k(e_j) \right) & = \frac{dx_1^k}{k!} \otimes e_j \in \CNF^{\otimes k} \otimes \mc E.
    \end{split}
\end{equation}
Thus, for every section $s=\sum_{j=1}^r f_j e_j \in \mathcal E^{\nabla}$, we have
\[
E(s)=\sum_{l=1}^r\left(\frac{\de^k f_l}{\de x_1^k}
+\sum_{j=1}^r\sum_{i=0}^{k-1}a_{ijl}(x_1)\frac{\de^i f_j}{\de x_1^i}\right)
\frac{dx_1^k}{k!}\otimes e_l.
\]
Hence, in local coordinates, finding a \emph{solution} $s\in \mcE^\nabla$ of $E$ is equivalent to finding local first integrals $f_1(x_1),\ldots,f_r(x_1)$ that are solutions of the system of differential equations
\begin{equation}\label{Eq: system of homogeneous linear differential equations}
    \frac{\de^k f_l}{\de x_1^k} + \sum_{j=1}^r\sum_{i=0}^{k-1} a_{ijl}(x_1) \cdot \frac{\de^i f_j}{\de x_1^i} =0, 1\le l \le r.
\end{equation}

\subsection{\texorpdfstring{Extensions of flat partial connections}{Extensions of flat partial connections}}

\begin{lemma}\label{L: system of differential equations}
Let $\F$ be a smooth codimension one foliation on a complex manifold $X$, and let $(\mcE,\nabla)$ be a flat partial connection on a locally free sheaf. Let $\sigma: \jetXF^k(\nabla) \rightarrow \CNF^{\otimes k} \otimes \mcE$ be a transverse differential equation, and let $E = \sigma \circ d^k: \mcE^\nabla \rightarrow \CNF^{\otimes k} \otimes \mcE$. Then, 
    \[
    d^{k-1}(\mcE^E) \subset \jetXF^{k-1}(\nabla)
    \]
    is a local system generating the sheaf $\jetXF^{k-1}(\nabla)$. Furthermore, $d^{k-1}(\mcE^E)$ determines a flat connection 
    \[
    \nabla_E: \jetXF^{k-1}(\nabla) \rightarrow \Omega_X^1(\jetXF^{k-1}(\nabla))
    \]
    that is an extension of the flat partial connection $\nabla^{k-1}$.
\end{lemma}
\begin{proof}
     Fix $x\in X$ and choose a sufficiently small neighborhood $U$ with coordinates $(x_1,\ldots,x_n)$ centered at $x$, where $\F$ is defined by $dx_1$ and $\mcE$ is free and admits a flat basis $\{e_1,\ldots,e_r\}$. Keeping the notation of \eqref{Eq: system of homogeneous linear differential equations}, we have that
    \[
     (\mcE^E)_x = \left\{\sum_{j=1}^r f_j\cdot e_j ; (f_1,\ldots,f_r) \text{ is a solution of System \eqref{Eq: system of homogeneous linear differential equations}} \right\}.
    \]
The existence and uniqueness theorem for homogeneous linear differential equations implies that the stalk $(\mcE^E)_x$ is isomorphic to $\C^{kr}$, where the isomorphism is given by 
    \begin{equation}\label{Eq: solutions of the dif equation in coordinates}
        \begin{split}
             (\mcE^E)_x & \longrightarrow  \C^{kr} \\
            \sum_{j=1}^r f_j \cdot e_j & \mapsto \left( \frac{\de^i f_j}{\de x_1^i}(0) \right)_{1\le j\le r, 0\le i \le k-1}.
        \end{split}
    \end{equation}
Observe that this isomorphism is exactly the local description of the evaluation at $x$ of $d^{k-1}$ restricted to $(\mcE^E)_x$ with respect to the basis $\mc B^{k-1}_{2,\nabla}$, that is, the composition of morphisms
 \begin{equation*}
        \begin{split}
             (\mcE^E)_x \xrightarrow{d^{k-1}}  \mc J_{X/\F,x}^{k-1}(\nabla) \xrightarrow{\pi} \jetXF^{k-1}(\nabla)(x) := \frac{\mc J^{k-1}_{X/\F,x}(\nabla)}{\m_x \cdot \mc J^{k-1}_{X/\F,x}(\nabla)}.
    \end{split}
\end{equation*}
Thus, $d^{k-1}(\mcE^E)_x \rightarrow \jetXF^{k-1}(\nabla)(x)$ is an isomorphism. Therefore, $ d^{k-1}(\mcE^E)$ is a local system generating the sheaf $\jetXF^{k-1}(\nabla)$.

For the second assertion, applying Corollary \ref{C: correspondence} for the foliation by points, there exists a flat connection $\nabla_E$ on the sheaf $\jetXF^{k-1}(\nabla)$ such that $(\jetXF^{k-1}(\nabla))^{\nabla_E} = d^{k-1}(\mcE^E)$, and since $(\jetXF^{k-1}(\nabla))^{\nabla_E} \subset  (\jetXF^{k-1}(\nabla))^{\nabla^{k-1}}$, it follows that $\nabla_E$ extends the flat partial connection $\nabla^{k-1}$.
\end{proof}

\begin{thm} Let $\F$ be a smooth codimension one foliation on a complex manifold $X$. Suppose that one of the following conditions holds:
\begin{enumerate}[label = -]
    \item $(\germe_X, \df)$ admits a transverse differential equation of order $k\ge 2$; or
    \item $(\CNF, \nabla_B)$ admits a transverse differential equation of order $k\ge 1$.
\end{enumerate}
Then $\F$ is a transversely affine foliation.
\end{thm}
\begin{proof}
Let $(\mcL,\nabla)$ be the line bundle with partial connection appearing in one of the two cases. By Lemma \ref{L: system of differential equations}, the partial connection $\nabla^{k-1}$ on $\jetXF^{k-1}(\nabla)$ has a flat extension $\nabla_E$. The short exact sequences in Corollary \ref{C: short exact sequence of jets of flat partial connections} give
\[
\det(\jetXF^{k-1}(\nabla),\nabla^{k-1})
\simeq(\CNF,\nabla_B)^{\otimes l}, 
\]
where
\[
l=\begin{cases}
k(k-1)/2,&\text{ if }(\mcL,\nabla)=(\germe_X,\df),\\
k(k+1)/2,&\text{ if }(\mcL,\nabla)=(\CNF,\nabla_B).
\end{cases}
\]
The hypotheses ensure that $l>0$. Taking determinants, we obtain a flat connection $\widetilde{\nabla}=\det\nabla_E$ on $\CNF^{\otimes l}$ extending $\nabla_B^{\otimes l}$.  Let us show how to construct the $l$-th root of $\widetilde{\nabla}$, which will be a connection on $\CNF$. On an open set with a local basis $\omega$ of $\CNF$, write $\widetilde{\nabla}(\omega^{\otimes l})=\eta\otimes \omega^{\otimes l}$ and define $\hatnabla(\omega)=(\eta/l)\otimes \omega$. If $\omega'=g \omega$, the connection form of $\widetilde{\nabla}$ changes from $\eta$ to $\eta+l\,dg/g$, so the local definitions of $\hatnabla$ are compatible and define a connection on $\CNF$. Because $\widetilde{\nabla}$ extends $\nabla_B^{\otimes l}$, the connection $\hatnabla$ extends $\nabla_B$. Finally, the curvature of $\widetilde{\nabla}$ is $l$ times the curvature of $\hatnabla$. Since $\widetilde{\nabla}$ is flat, so is $\hatnabla$. Therefore $\F$ is transversely affine.
\end{proof}

\begin{rmk}
    The existence of a transverse differential equation of order $1$ on $(\germe_X,\df)$ does not imply that $\F$ is transversely affine. Indeed, since $\jetXF^1(\df) = \CNF \oplus \germe_X$, the trivial splitting of the exact sequence
    \[
    0 \rightarrow \CNF \rightarrow \jetXF^1(\df) \rightarrow \germe_X \rightarrow 0,
    \]
    always exists, and it corresponds to the differential equation  
    \[
    E(f) = \frac{\de f}{\de x_1}, f\in \germe_{X/\F},
    \]
whose solutions are exactly the constant functions. Nevertheless, a non-trivial splitting of the exact sequence above corresponds to a horizontal morphism $(\germe_X,\df) \rightarrow (\CNF, \nabla_B)$, and this corresponds to a global closed holomorphic 1-form $\omega$ defining $\F$. Therefore, in this case, we also conclude that $\F$ admits a singular transverse affine structure.
\end{rmk}

\subsection{First order differential equations and \texorpdfstring{transverse}{transverse} affine structures}
\begin{thm}\label{T: flat meromorphic connections, bijection}
    Let $\F$ be a smooth codimension one foliation on a complex manifold $X$, and let $(\mcE,\nabla)$ be a flat partial connection on a locally free sheaf. Then, there exists a natural bijection between:
\begin{enumerate}[label = (\roman*)]
    \item flat extensions $\hatnabla: \mcE \rightarrow \Omega_X^1(\mcE)$ of $\nabla$; and 
    \item horizontal splittings of the short exact sequence 
    \begin{equation}\label{Eq: short exact sequence of the first jet of flat partial connections}
        0 \rightarrow (\CNF \otimes \mcE, \nabla_B \otimes \nabla) \xrightarrow{\iota} (\jetXF^1(\nabla), \nabla^1) \xrightarrow{\pi} (\mcE,\nabla) \rightarrow 0.
    \end{equation}
\end{enumerate}
\end{thm}
\begin{proof}
    By Lemma \ref{L: system of differential equations}, we have already described how a horizontal splitting of the short exact sequence \eqref{Eq: short exact sequence of the first jet of flat partial connections}, which is the same as a transverse differential equation $E: \mcE^\nabla \rightarrow \CNF \otimes \mcE$ of order 1, defines a flat extension $\nabla_E$ of $\nabla$. Let us describe the converse construction.
    
    Let $\hatnabla: \mcE \rightarrow \Omega_X^1(\mcE)$ be a flat extension of $\nabla$, and let $\sigma: \jetX^1(\mcE) \rightarrow \Omega_X^1(\mcE)$ be the corresponding splitting of the short exact sequence of $\jetX^1(\mcE)$. For every $s\in \mcE^\nabla$, we have
    \[
    \sigma(d^1(s)) = \hatnabla(s) \in \CNF \otimes \mcE = \ker(\mrrestr: \Omega_X^1(\mcE) \rightarrow \CTF^1(\mcE)),
    \]
    because $\hatnabla$ extends $\nabla$. Since $\jetXF^1(\nabla)$ is the $\germe_X$-module generated by the first jets of flat sections, it follows that $\sigma$ induces a horizontal $\germe_X$-linear morphism $\sigma: \jetXF^1(\nabla) \rightarrow \CNF \otimes \mcE$ such that $\sigma \circ \iota = \mathrm{id}$, that is, a splitting of the short exact sequence \eqref{Eq: short exact sequence of the first jet of flat partial connections}.

    Finally, starting with $\hatnabla$, the corresponding splitting $\sigma: \jetXF^1(\nabla) \rightarrow \CNF \otimes \mcE$ is such that 
    \[
    \ker(\sigma \circ d^1) = \{s\in \mcE^\nabla ; \sigma(d^1(s))=0  \} = \{s\in \mcE^\nabla; \hatnabla(s)=0 \} = \mcE^{\hatnabla},
    \]
    and hence the described constructions are inverses of each other. Therefore, we have established a bijection. This concludes the proof.
\end{proof}

\begin{cor}\label{C: transversely affine structures}
    Let $\F$ be a smooth codimension one foliation on a complex manifold $X$. Then, there exists a natural bijection between:
\begin{enumerate}[label = (\roman*)]
    \item transverse affine structures; and 
    \item horizontal splittings of the short exact sequence associated to the first transverse jet of $(N_{\F}, \nabla_B)$:
    \[
    0 \rightarrow (\germe_X,\df) \xrightarrow{\iota} (\jetXF^1(\nabla_B),\nabla_B^1) \xrightarrow{\pi} (N_{\F},\nabla_B) \rightarrow 0.
    \]
\end{enumerate}
\end{cor}

Note that in the short exact sequence above, we are using the isomorphism $(\CNF,\nabla_B) \otimes (N_{\F},\nabla_B) \simeq (\germe_X, \df)$.

\subsection{Second order differential equations and \texorpdfstring{transverse}{transverse} projective structures}
This section is completely based on \cite[Chapter I, Section 5]{Deligne-70-zbMATH03385791}. We aim to generalize \cite[Chapter I, Proposition 5.12]{Deligne-70-zbMATH03385791} for codimension one smooth foliations.

Let $\F$ be a smooth codimension one foliation on a complex manifold $X$, and let $(\mcL, \nabla)$ be a flat partial connection on a line bundle. Let us explain how a second order transverse differential equation on $(\mcL,\nabla)$ naturally leads to both a transverse projective structure $\mc P$ for $\F$, and a flat extension of the connection $\nabla_B \otimes \nabla^{\otimes 2}$ on $\CNF \otimes \mcL^{\otimes 2}$. 

We start with the second piece of data, which is easier. By Lemma \ref{L: system of differential equations}, a second order transverse differential equation $\sigma: \jetXF^2(\nabla) \rightarrow \CNF^{\otimes 2} \otimes \mc L$ determines an extension $\nabla_E$ of $\nabla^1$ on $\jetXF^1(\nabla)$, where $E = \sigma \circ d^2:  \mcL^\nabla \rightarrow \CNF^{\otimes 2} \otimes \mc L$ and
\[
   (\jetXF^1(\nabla))^{\nabla_E} = \{d^1(s); s \in \mcL^E \}.
\]
Taking the determinant, $\det \nabla_E$ is a flat connection on $\det \jetXF^1(\nabla) \simeq \CNF \otimes \mcL^{\otimes 2}$ that extends the connection $\nabla_B \otimes \nabla^{\otimes 2}$. Consider the following claim (that will be useful in the proof of Theorem \ref{T: second order transverse differential equations} below):

\begin{claim}\label{C: connection in coordinates}
    Let $U\subset X$ be an open subset of $X$ with a foliated system of coordinates $(x_1,\ldots,x_n)$, $\F$ defined by $dx_1$, and such that $\restr{\mcL}{U}$ is free with a flat basis $s\in \mcL$. Suppose, in these coordinates, that the second order differential equation $E$ is given by
    \[
    E(f) = f''+ a(x_1) \cdot f' + b(x_1)f.
    \]
Let $\det\nabla_E$ be the flat connection on $\det\restr{\jetXF^1(\nabla)}{U} \simeq \restr{\CNF\otimes\mc L^{\otimes 2}}{U}$ given by Lemma \ref{L: system of differential equations}. Then,
    \[
    (\det \nabla_E) (dx_1 \otimes s^{\otimes 2}) = a(x_1)dx_1 \otimes (dx_1 \otimes s^{\otimes 2}) \in \Omega_X^1(\CNF \otimes \mc L^{\otimes 2}).
    \]
\end{claim}
\begin{proof} In the local basis $\{dx_1 \otimes s, 1\otimes s\}$ for $\jetXF^1(\nabla)$, a horizontal section corresponding to a solution $f$ has coordinate column $(f',f)^T$. The equation $E(f)=0$ gives
    \[
    \frac{d}{dx_1}\left(\begin{array}{c} f' \\ f \end{array} \right)
    = \left(\begin{array}{cc} -a(x_1) & -b(x_1) \\ 1 & 0 \end{array} \right)
    \left(\begin{array}{c} f' \\ f \end{array} \right).
    \]
    Note that this is simply the companion system associated to $E$ (see \cite[Chapter I, Equation 4.8.1]{Deligne-70-zbMATH03385791}). Hence, $\nabla_E$ is given by
    \[
    \nabla_E\left(\begin{array}{c} f_1 \\ f_2 \end{array} \right)
    = \left(\begin{array}{c} df_1 \\ df_2 \end{array} \right)
    + \left(\begin{array}{cc} a(x_1) & b(x_1) \\ -1 & 0 \end{array} \right)dx_1
    \left(\begin{array}{c} f_1 \\ f_2 \end{array} \right).
    \]
    Taking its trace gives the connection $\det \nabla_E$ on $dx_1\otimes s^{\otimes 2}$.
\end{proof}

It remains to construct a transverse projective structure associated to $\sigma$. For every $x\in X$, let $s_1,s_2 \in (\mcL^E)_x$ be sections such that $d^1s_1,d^1s_2$ are linearly independent sections of the local system $d^1(\mcL^E)$ at $x$, and consider the well-defined map $\phi = (s_1:s_2): U \rightarrow \Pj^1$,  where $U$ is a sufficiently small neighborhood of $x$ on which $s_1$ and $s_2$ are defined and have no common zero. 

\begin{claim}
    The map $\phi$ is a local submersion defining $\F$.
\end{claim}
\begin{proof}
    Consider, on a neighborhood of $x$, a foliated system of coordinates $(x_1,\ldots,x_n)$ with $x_1$ defining $\F$, and a flat basis $s\in \mc L$. Writing $s_i = f_i(x_1) \cdot s$, suppose, without loss of generality, that $f_2(0)\neq 0$. Hence, we have $\phi = \frac{f_1}{f_2}$ and
    \[
    \phi'(0) = \frac{f_1'(0) \cdot f_2(0) - f_2'(0)\cdot f_1(0)}{f_2(0)^2} \neq 0,
    \]
    because $d^1(s_1) = (f_1(0),f_1'(0))$ and $d^1(s_2) = (f_2(0),f_2'(0))$ are linearly independent (here we are using the isomorphism of \eqref{Eq: solutions of the dif equation in coordinates}). Therefore, $\phi$ is a submersion. Since $\phi=f_1(x_1)/f_2(x_1)$ depends only on the transverse coordinate $x_1$ and has nonzero derivative, its local fibers are precisely the leaves of $\F$. 
\end{proof}

Different choices of linearly independent pairs of solutions may yield different local submersions $\tilde{\phi}=(\tilde{s}_1:\tilde{s}_2)$. Thus, to determine a transverse projective structure for $\F$, we must verify that the respective submersions $\phi,\tilde{\phi}$ differ by an automorphism of $\Pj^1$. In order to prove that, we need the concept of the Schwarzian derivative.

For every holomorphic homeomorphism $f :U\rightarrow V$ between open sets of $\C$, we define the \emph{Schwarzian derivative} of $f$ by
\[
\Theta(f) := \frac{f' \cdot (f'''/6) - (f''/2)^2}{(f')^2},
\]
see \cite[Chapter I, Equation 5.9.2]{Deligne-70-zbMATH03385791}. It is well known that the Schwarzian derivative satisfies the following equation (see \cite[Lemma 24]{gunning-66-zbMATH03280321})
\[
    \Theta(f_1 \circ f_2) = \Theta(f_1) \circ f_2 \cdot (f_2')^2 + \Theta(f_2),
\]
and that $\Theta(f)=0$ if and only if $f$ is the restriction of an automorphism of $\Pj^1$ \cite[p. 166]{gunning-66-zbMATH03280321}. 

\begin{claim}\label{C: schwarzian in coordinates}
    Let $U\subset X$ be an open subset of $X$ with a foliated system of coordinates $(x_1,\ldots,x_n)$, $\F$ defined by $dx_1$, and such that $\restr{\mcL}{U}$ is free with a flat basis $s\in \mcL$. Suppose, in these coordinates, that the second order differential equation $E$ is given by
    \[
    E(f) = f''+ a(x_1) \cdot f' + b(x_1)f.
    \]
    Then, for every pair $f_1,f_2$ of solutions of $E$ with $d^1(f_1),d^1(f_2)$ linearly independent,
    \[
    \Theta(f_1:f_2) = \frac{1}{3} \cdot b -\frac{1}{12} \cdot (a^2 +2a').
    \]
\end{claim}
\begin{proof}
    The calculations can be made using the explicit description of the Schwarzian derivative of a map $\phi= (g:h): U \rightarrow \Pj^1$ given by \cite[Chapter 1, Equation 5.9.3] {Deligne-70-zbMATH03385791}. See \cite[Chapter 1, Proof of Proposition 5.12]{Deligne-70-zbMATH03385791} for those explicit calculations.
\end{proof}

Let us use Claim \ref{C: schwarzian in coordinates} to conclude that
\[
    \mc P_{E} := \{\phi=(s_1:s_2): U \rightarrow \Pj^1 ; s_1,s_2 \in \mcL^E, d^1s_1,d^1s_2 \text{ linearly independent }\}
\]
induces a transverse projective structure for $\F$. Indeed, let two submersions $\phi = (s_1:s_2)$ and $\tilde{\phi} = (\tilde{s}_1:\tilde{s}_2)$ be defined on the same open subset $U\subset X$. Since both $\phi, \tilde{\phi}$  define $\F$, we have $\phi = \psi \circ \tilde{\phi}$ for some germ of biholomorphism $\psi$. Since, by Claim \ref{C: schwarzian in coordinates}, $\Theta(\phi) = \Theta(\tilde{\phi})$, it follows that $\Theta(\psi)=0$, and therefore $\psi \in \Aut(\Pj^1)$. We call $\mc P_{E}$ the \emph{transverse projective structure associated to $E$}.

Summarizing: starting with a transverse differential equation of second order $E$, we construct a transverse projective structure $\mc P_{E}$ and an extension $\det(\nabla_E)$ of the partial connection $\nabla_B \otimes \nabla^{\otimes 2}$ on $\CNF \otimes \mcL^{\otimes 2}$. The following theorem states that this process can be reversed:

\begin{thm}\label{T: second order transverse differential equations}
    Let $\F$ be a smooth codimension one foliation on a complex manifold $X$. Let $(\mcL, \nabla)$ be a flat partial connection on a line bundle. Then, there exists a natural bijection between:
    \begin{enumerate}[label = (\roman*)]
        \item transverse differential equations of second order on $(\mcL, \nabla)$; and
        \item pairs $(\mc P, \hatnabla)$, where $\mc P$ is a transverse projective structure for $\F$, and $\hatnabla$ is a flat extension of the connection $\nabla_B \otimes \nabla^{\otimes 2}$ on $\CNF \otimes \mc L^{\otimes 2}$.
    \end{enumerate}
\end{thm}

\begin{proof}
    We already described how a transverse differential equation of second order $E$ induces the pair $(\mc P, \hatnabla)$. To prove the converse correspondence, we construct the unique transverse differential equation locally from the data of $(\mc P, \hatnabla)$. By uniqueness, these local equations glue to give a global transverse differential equation of second order.

    Suppose we have the pair $(\mc P, \hatnabla)$. Let $U\subset X$ be an open subset with a foliated system of coordinates $(x_1,\ldots,x_n)$, $\F$ defined by $dx_1$, and with $s\in \mc L$ a flat basis for $\mc L$. With respect to these local coordinates, we calculate that
    \[
    \hatnabla (dx_1 \otimes s^{\otimes 2}) = a(x_1)dx_1 \otimes (dx_1 \otimes s^{\otimes 2}),
    \]
    and for any $\phi: U \rightarrow \Pj^1$ in the transverse projective structure $\mc P$, we calculate that
    \[
    \Theta(\phi) = c(x_1).
    \]
    Let $b(x_1) := 3c(x_1) + 1/4\cdot (a(x_1)^2 + 2a'(x_1))$. On the open subset $U\subset X$, we define the second order transverse differential equation $\sigma_U: \restr{\jetXF^2(\nabla)}{U} \rightarrow \restr{\CNF^{\otimes 2} \otimes \mcL}{U}$ corresponding to
    \[
    E(f) = f'' + a(x_1) \cdot f' + b(x_1)f.
    \]
    By Claims \ref{C: connection in coordinates} and \ref{C: schwarzian in coordinates}, this is the only second order transverse differential equation that recovers the connection $\hatnabla$ and the projective structure $\mc P$ on $U$.

    Therefore, the local second order transverse differential equations $\{\sigma_U\}$ agree on overlaps, and thus we define globally a second order transverse differential equation $\sigma: \jetXF^2(\nabla) \rightarrow \CNF^{\otimes 2} \otimes \mc L$. This concludes the proof.
\end{proof}

\section{Prolongation of \texorpdfstring{transverse}{transverse} projective structures}\label{Section: prolongation}

\subsection{Jet bundles}

Let $\pi:E \rightarrow X$ be a vector bundle over a complex manifold $X$, and let $\mcE$ be the sheaf of sections of $E$. We define the $k$-th jet bundle of $E$, denoted by $J^k_XE$, as the vector bundle defined as
\[
(J^k_XE)_x = \frac{\left\{s: (U,x) \rightarrow E \text{ germ of local section of } \pi:E\rightarrow X \right\}}{\sim_k}, \forall x \in X,
\]
where $\sim_k$ is defined as follows. Let $(x_1,\ldots,x_n)$ be a system of coordinates in a neighborhood of $x$, and let $\{e_1,\ldots,e_r\}$ be a basis for $\mcE$ in a neighborhood of $x$; let $s = \sum_{j=1}^r f_j \cdot e_j$ and $s' = \sum_{j=1}^r f_j' \cdot e_j$ germs of local sections of $E$. We say that $s \sim_k s'$ if
\[
\frac{\de^{\mbf i} f_j}{\de \mbf x^{\mbf i}}(x) = \frac{\de^{\mbf i} f_j'}{\de \mbf x^{\mbf i}}(x), 1\le j \le r, |\mbf i| \le k.
\]
As we pointed out in Remark \ref{Rmk: different definitions of jets}, the jet bundle $J^k_XE$ is the total space of the sheaf of jets $\jetX^k(\mcE)$ with respect to the canonical $\germe_X$-module structure. We refer to \cite[Chapter 6]{saunders-89-zbMATH00042050} for a detailed discussion on the properties of jet bundles.

\subsection{Prolongation of morphisms of vector bundles}  
Let $\phi: X \rightarrow Y$ be an isomorphism of complex manifolds. Let $E$ be a vector bundle over $X$, and let $E'$ be a vector bundle over $Y$. Let $\psi: E \rightarrow E'$ be a bundle morphism. 
For every $k\ge 0$, we define the \emph{$k$-th prolongation of } $\psi$ as the bundle morphism $\prolong{\psi}{k}: J^k_X E \rightarrow J^k_Y E'$ such that, for every local section $s: U\rightarrow E$, 
\[
\prolong{\psi}{k} \circ (d^ks) = d^k(s') \circ \phi,
\]
where $s': \phi(U) \rightarrow E'$ is the section of $E'$ such that $\psi \circ s =s'\circ \phi$ (see \cite[Definition 6.2.17]{saunders-89-zbMATH00042050}).

\begin{ex}\label{Ex: example of the first prolongation of the pushfoward}
    Let $\phi: X \rightarrow Y$ be an isomorphism of complex manifolds, and let $d\phi: TX \rightarrow TY$ be the pushforward of vector fields. Let us describe $\prolong{d\phi}{1}: J_X^1(TX) \rightarrow J_Y^1(TY)$ in local coordinates. Let $\mbf x = (x_1,\ldots,x_n)$ be a system of coordinates for $X$, let $\mbf y = (y_1,\ldots,y_n)$ be a system of coordinates for $Y$, and let $\phi = (\phi_1,\ldots,\phi_n)$ be the description of $\phi$ in these coordinates. Considering the natural coordinate frame $\left\{ \de/\de x_1,\ldots, \de /\de x_n \right\}$ for $TX$, there exists a natural system of coordinates $(\mbf z,\mbf w)= (\{z_j\},\{w_{j_1j_2}\}, 1\le j,j_1,j_2 \le n \})$ for $J_X^1(TX)$ such that
    \[
        z_j\left(d^1\left(\sum_{k=1}^n f_k \cdot \frac{\de}{\de x_k} \right)\right) = f_j,\text{ and }
        w_{j_1j_2}\left(d^1\left(\sum_{k=1}^n f_k \cdot \frac{\de}{\de x_k} \right)\right) = \frac{\de f_{j_2}}{\de x_{j_1}}.
    \]
    See \cite[Definition 4.1.5]{saunders-89-zbMATH00042050}. Similarly, taking $\left\{ \de/\de y_1,\ldots, \de /\de y_n \right\}$ as the natural frame for $TY$, there exists a natural system of coordinates $(\mbf z',\mbf w')$. With respect to the coordinates $(\mbf x, \mbf z, \mbf w)$ and $(\mbf y, \mbf z',\mbf w')$, a straightforward calculation shows that 
    
    \[
        \prolong{(d\phi)}{1}(\mbf x, \mbf z,\mbf w) = (\phi(\mbf x), \ldots ,\prolong{\phi}{1}_j(\mbf x, \mbf z), \ldots, \prolong{\phi}{2}_{j_1j_2}(\mbf x, \mbf z, \mbf w),\ldots),
    \]
    where
    \[
        \phi^{(1)}_j(\mbf x, \mbf z) = \sum_{i=1}^n \frac{ \de \phi_j}{\de x_i}(\mbf x) \cdot z_i, 1\le j \le n,
    \]
    and
    \begin{multline*}
        \phi^{(2)}_{j_1j_2}(\mbf x,\mbf z,\mbf w) = \sum_{i_1,i_2=1}^n \left(\frac{\de (\phi^{-1})_{i_1}}{\de y_{j_1}} \circ \phi(\mbf x) \cdot \frac{\de \phi_{j_2}}{\de x_{i_2}}(\mbf x)\right) \cdot w_{i_1i_2} \\ + \sum_{i=1}^n\left(\sum_{k=1}^n\frac{\de (\phi^{-1})_{k}}{\de y_{j_1}} \circ \phi(\mbf x) \cdot \frac{\de^2 \phi_{j_2}}{\de x_{k}\de x_{i}}(\mbf x)  \right)\cdot z_i, 1\le j_1,j_2\le n.
    \end{multline*}
    These calculations will be useful in the proof of Lemma \ref{L: trivial isotropy}.
\end{ex}

\subsection{Prolongation of foliations}

Let $\F$ be a smooth foliation on a complex manifold $X$. Let $(N_{\F},\nabla_B)$ be the Bott connection on the normal sheaf of the foliation. For every $k\ge 0$, consider the $k$-th sheaf of transverse jets $(\jetXF^k(\nabla_B),\nabla_B^{k})$. We denote the total space of $\jetXF^k(\nabla_B)$ by $\prolong{X_{\F}}{k+1}$, and the foliation on $\prolong{X_{\F}}{k+1}$ induced by $\nabla_B^k$ by $\prolong{\F}{k+1}$. We call the pair $(\prolong{X_{\F}}{k+1}, \prolong{\F}{k+1})$ the $(k+1)$-th prolongation of the foliation. We start the indexing at $1$ so that $\prolong{X_{\F}}{1}=E(N_{\F})$, while the index $0$ refers to the original manifold $X$.

For the foliation by points on a complex manifold $Y$, we set $\prolong{Y}{k}:=E(J_Y^{k-1}(T_Y))$ for $k\ge1$ and equip this space with the foliation by points. This agrees with the preceding indexing convention.

Let $\phi:X\to Y$ be a submersion defining $\F$. For every $k\ge1$, we construct a submersion $\prolong{\phi}{k}:\prolong{X_{\F}}{k}\to\prolong{Y}{k}$ whose connected fiber components are the leaves of $\prolong{\F}{k}$. Since $\ker d\phi=T_{\F}$, the differential induces an $\germe_X$-linear isomorphism $d\phi:N_{\F}\to\phi^*T_Y$. Let $\nabla_Y$ be the flat partial connection on $\phi^*T_Y$ whose horizontal sections are $\phi^{-1}T_Y$ (Example \ref{Ex: partial connection induced by submersion}). The isomorphism $d\phi$ is horizontal and therefore induces the horizontal isomorphism

\[
\prolong{d\phi}{k-1}:(\jetXF^{k-1}(\nabla_B),\nabla_B^{k-1}) \longrightarrow(\jetXF^{k-1}(\nabla_Y),\nabla_Y^{k-1}).
\]
The natural inclusion $\phi^*\jetY^{k-1}(T_Y)\hookrightarrow\jetX^{k-1}(\phi^*T_Y)$ identifies the transverse jet $\jetXF^{k-1}(\nabla_Y)$ with $\phi^*\jetY^{k-1}(T_Y)$. In local coordinates adapted to $\phi$, it is easy to verify that

\[
(\jetXF^{k-1}(\nabla_Y))^{\nabla_Y^{k-1}}=\phi^{-1}\jetY^{k-1}(T_Y).
\]
Consequently, the total space $E(\jetXF^{k-1}(\nabla_Y))$ is naturally $X\times_Y\prolong{Y}{k}$. The map $\prolong{\phi}{k}$ is the composition
\[
\prolong{X_{\F}}{k}\rightarrow X\times_Y\prolong{Y}{k}\longrightarrow\prolong{Y}{k}.
\]
The last arrow is a submersion, and the connected components of its fibers are the leaves induced by $\nabla_Y^{k-1}$. Transporting them through $\prolong{d\phi}{k-1}$ gives the leaves of $\prolong{\F}{k}$.

\subsection{The second prolongation of projective spaces}\label{S: the second prolongation of the projective space}

Let us fix some notation for this section. We denote by $G$ the group $\PSL(n+1,\C)$, the automorphism group of the projective space $\Pj^n$, and by $\g$ the Lie algebra of $G$. By definition, $G$ acts on $\Pj^n$, and for every $g\in G$, we denote by $L_g: \Pj^n \rightarrow \Pj^n$ the action of $g$ on $\Pj^n$.

For every $g\in G$, the isomorphism $L_g: \Pj^n \rightarrow \Pj^n$ can be prolonged, inducing an isomorphism of vector bundles $\prolong{L_g}{2}: \prolongPjntwo \rightarrow \prolongPjntwo$. Furthermore, since $L_{g_2} \circ L_{g_1} = L_{g_2 \cdot g_1}$, it follows that
\[
 \prolong{L_{g_2}}{2} \circ \prolong{L_{g_1}}{2} = \prolong{L_{g_2\cdot g_1}}{2}, \forall g_1,g_2\in G.
\]
Thus, the action of $G$ on $\Pj^n$ induces an action of $G$ on $\prolongPjntwo$. For every $\xi\in \prolongPjntwo, g\in G$, we denote  $g \cdot \xi:=\prolong{L_g}{2}(\xi)$.

\begin{lemma}\label{L: trivial isotropy}
    Using the notation above, for a generic point $\xi\in \prolongPjntwo$, the isotropy group $G_\xi$ is trivial.
\end{lemma}

 The proof of Lemma \ref{L: trivial isotropy} is given in Subsection \ref{S: proof of the action}. The total space $\prolongPjntwo$ is irreducible because it is a vector bundle over $\Pj^n$, and it has dimension $n+n(n+1)=(n+1)^2-1=\dim G$. If $G_\xi$ is trivial, the orbit $G\cdot\xi$ has dimension $\dim G$. Its closure has the same dimension as $\prolongPjntwo$, and thus the orbit is dense. By \cite[Proposition 8.3]{humphreys-zbMATH03508744}, the orbit is locally closed, so it is open in its closure. Since its closure is all of $\prolongPjntwo$, the orbit is open. Distinct orbits are disjoint, whereas any two nonempty open subsets of an irreducible variety intersect. Thus there is at most one orbit with trivial stabilizer, and Lemma \ref{L: trivial isotropy} shows that it exists. We denote this unique open orbit by $G^\circ\subset\prolongPjntwo$.

\begin{prop}\label{P: birational morphism} For every $\xi\in G^\circ$, the map
\begin{align*}
    \phi_\xi: G & \longrightarrow \prolongPjntwo \\
    g & \mapsto g \cdot \xi
\end{align*}
is a birational morphism. Additionally, for every $\xi\in G^\circ$, $G \cdot \xi = G^\circ$.
\end{prop}
\begin{proof}
For every $\xi\in \prolongPjntwo$, the orbit $G \cdot \xi \subset \prolongPjntwo$ is a subvariety of $\prolongPjntwo$ (see \cite[Proposition 8.3]{humphreys-zbMATH03508744}). By Lemma \ref{L: trivial isotropy}, for a generic point $\xi\in \prolongPjntwo$, the isotropy group $G_\xi$ is trivial, and thus $\dim G \cdot \xi = \dim G = \dim \prolongPjntwo$. Therefore, $G \cdot \xi$ contains a Zariski open subset of $\prolongPjntwo$, that is, $\phi_\xi$ is dominant. Finally, since $\phi_\xi$ is injective, we conclude that $\phi_\xi$ is a birational regular morphism.  

For the second claim, since $\prolongPjntwo$ is irreducible, it follows that $G\cdot \xi \cap G\cdot \xi'\neq \emptyset$ for every pair of points $\xi,\xi'\in G^\circ$. Thus, the orbits are the same, and therefore they must be $G^\circ$.
\end{proof}

\begin{rmk}\label{rmk: complement of G circ}
The complement of the open orbit $G^\circ\subset\prolongPjntwo$ is an algebraic divisor. Indeed, in the coordinates $(Z,W)$ on a fiber used in the proof of Lemma \ref{L: trivial isotropy}, put
\[
\Delta(Z,W)=\det\begin{pmatrix}Z^T\\ Z^TW\\ \vdots\\ Z^TW^{n-1}\end{pmatrix}.
\]
The action formula in that proof shows that, under an element of the stabilizer of the base point with linear part $A$, the matrix above is multiplied on the left by a lower triangular matrix with diagonal entries $1$ and on the right by $A^T$. Thus $\Delta$ is multiplied by $\det A$. Since $G$ acts transitively on $\Pj^n$, the condition $\Delta\neq0$ defines a $G$-invariant open subset of $\prolongPjntwo$. This subset is nonempty, and its points have trivial stabilizer by the proof of Lemma \ref{L: trivial isotropy}, that is, $\{\Delta\neq 0\} \subset G^\circ$, and its invariance implies we have equality. Hence $G^\circ$ is exactly the locus $\Delta\neq0$, and its complement is the divisor locally defined by $\Delta=0$.
\end{rmk}

Let $\Omega_G: T_G \rightarrow \g \otimes \germe_G$ be the Maurer--Cartan form on $G$ invariant under left multiplication by $G$ (see \cite[Chapter 3, Definition 1.3]{sharpe-97-zbMATH00914851}). For every $\xi\in G^\circ$, the birational map $\phi_\xi^{-1}:\prolongPjntwo \dashrightarrow G$ induces a rational $\g$-valued 1-form 
\begin{equation*}
    \Omega_\xi: T_{\prolongPjntwo} \to \g \otimes \germe_{\prolongPjntwo}(D_\xi),
\end{equation*}
for some effective divisor $D_\xi$ in $\prolongPjntwo$, transverse to the fibers of the fibration $\prolongPjntwo\rightarrow \Pj^n$. Observe that $\Omega_\xi$ is invariant under the action of $G$ on $\prolongPjntwo$. Indeed, since the diagram
\begin{equation*}
    \begin{tikzpicture}
        \matrix(m)[matrix of math nodes, column sep = 2em, row sep = 2em]
        {
        G & \prolongPjntwo \\
        G & \prolongPjntwo \\
        };
        \path[->]
        (m-1-1) edge node[above]{$\phi_\xi$} (m-1-2) edge node[left]{$L_g$} (m-2-1)
        (m-1-2) edge node[right]{$\prolong{L_g}{2}$} (m-2-2)
        (m-2-1) edge node[above]{$\phi_\xi$} (m-2-2)
        ;
    \end{tikzpicture}
\end{equation*}
commutes, it follows that
\begin{equation}\label{Eq: psl structure is invariant by the action}
    (\prolong{L_g}{2})^*(\Omega_\xi) = (\prolong{L_g}{2})^* \circ (\phi_\xi^{-1})^*\Omega_G = (\phi_\xi^{-1})^* \circ (L_g)^*\Omega_G = (\phi_\xi^{-1})^*\Omega_G = \Omega_\xi.
\end{equation}
Hence, for every $\xi\in G^\circ$, we have defined a $\g$-valued 1-form invariant by the action of $G$. Although the collection of 1-forms $\Omega_\xi$ depends on $\xi$, its members satisfy the following relation.

\begin{prop}\label{P: rational g valued 1-form depends on q}
    Let $\xi,\xi' \in G^\circ$, and let $g \in G$ such that $\xi' = g \cdot \xi$. Let $\Ad: G \rightarrow \Aut(\g)$ be the adjoint representation of the group $G$ on the Lie algebra $\g$. Then,
    \begin{equation*}
        \Omega_{\xi'} = \Ad(g) \circ \Omega_{\xi}.
    \end{equation*}
\end{prop}
\begin{proof}
    Since $\phi_{\xi'} = \phi_\xi \circ R_g$, we have that
    \begin{align*}
        \Omega_{\xi'} & = (\phi_{\xi'}^{-1})^*(\Omega_G) = (R_{g^{-1}} \circ \phi_\xi^{-1})^*(\Omega_G) = (\phi_\xi^{-1})^*(\Ad(g) \circ \Omega_G) \\
        & = \Ad(g) \circ (\phi_\xi^{-1})^*(\Omega_G) = \Ad(g) \circ \Omega_\xi,
    \end{align*}
    and this concludes the proof.
\end{proof}

This proposition has two immediate consequences. The first one is that the polar divisor $D_\xi$ of $\Omega_\xi$ does not depend on $\xi\in G^\circ$. The second one is that, for every $\xi\in G^\circ$, the $\g$-valued form $\Omega_\xi$ defines the same rational $G$-structure on $\prolongPjntwo$, up to some automorphism of $G$. Here we regard structures related by such an automorphism as the same, transporting the model subgroup $H$ by conjugation whenever a subgroup is specified.

\begin{prop}\label{P: psl structure of the prolongation of the projective space, restriction to the relative tangent bundle}
    Let $\xi\in G^\circ$, and let $\eta=\pi(\xi) \in \Pj^n$ be the projection of $\xi$ to $\Pj^n$. Let $\h_\eta$ be the subalgebra of $\g$ corresponding to the isotropy group $G_\eta \subset G$. Let $\Omega_\xi: T_{\prolongPjntwo} \rightarrow \g \otimes \germe_{\prolongPjntwo}(D_\xi)$ be the rational $\g$-valued 1-form induced by the birational morphism $\phi_\xi: G \rightarrow G^\circ$. Then, the restriction of $\Omega_\xi$ to the relative tangent bundle $T_{\prolongPjntwo/\Pj^n} \subset T_{\prolongPjntwo}$ factors through the inclusion $\h_\eta \subset \g$, that is, it induces the morphism
        \begin{equation}\label{Eq: restriction to the relative tangent bundle}
            \Omega_\xi: T_{\prolongPjntwo/\Pj^n} \rightarrow \h_\eta \otimes \germe_{\prolongPjntwo}(D_\xi).
        \end{equation}
\end{prop}
\begin{proof}
Let $\rho_\eta:G\rightarrow\Pj^n$ be the map $g\mapsto g\cdot\eta$. Since $\pi\circ\phi_\xi=\rho_\eta$, the inverse $\phi_\xi^{-1}:G^\circ\rightarrow G$ respects these fibrations and induces
    \[
        d(\phi_\xi^{-1}):T_{G^\circ/\Pj^n}\rightarrow T_{G/\Pj^n}.
    \]
At $g\in G$, the relative tangent space of $\rho_\eta$ is $d(L_g)_e(\h_\eta)$. The left Maurer--Cartan form $\Omega_G$ maps this space into $\h_\eta$ (see \cite[Chapter 3, Definition 1.3]{sharpe-97-zbMATH00914851}). Hence $\Omega_\xi=(\phi_\xi^{-1})^*\Omega_G$ maps $T_{G^\circ/\Pj^n}$ into $\h_\eta$. The resulting map to $\g/\h_\eta$ vanishes on the dense open subset $G^\circ$ and hence vanishes as a rational morphism on $\prolongPjntwo$. Therefore,  the restriction of $\Omega_\xi$ to $T_{\prolongPjntwo/\Pj^n} \subset T_{\prolongPjntwo}$ factors through $\h_\eta$.
\end{proof}

\begin{rmk}
    By \eqref{Eq: restriction to the relative tangent bundle}, the restriction of $\Omega_\xi$ to the relative tangent bundle depends on the point $\eta=\pi(\xi)\in \Pj^n$. Furthermore, it also depends on $\xi\in G^\circ$. Indeed, let $\xi,\xi' \in G^\circ$ such that $\eta=\pi(\xi)=\pi(\xi') \in \Pj^n$, and let $h\in G_\eta$ such that $\xi' = h \cdot \xi$. By Proposition \ref{P: rational g valued 1-form depends on q}, $\Omega_\xi' = \Ad(h) \circ \Omega_\xi$ and thus we have the commutative diagram:
    \begin{equation*}
        \begin{tikzpicture}
            \matrix(m)[matrix of math nodes, column sep = 2em, row sep = 2em]
            {
            T_{\prolongPjntwo/\Pj^n} & \h_\eta \otimes \germe_{\prolongPjntwo}(D_\xi) \\
             & \h_\eta \otimes \germe_{\prolongPjntwo}(D_\xi) \\
            };
            \path[->]
            (m-1-1) edge node[above]{$\Omega_\xi$} (m-1-2) edge node[below]{$\Omega_{\xi'}$} (m-2-2)
            (m-1-2) edge node[right]{$\Ad(h)$} (m-2-2)
            ;
        \end{tikzpicture},
    \end{equation*}
    where now $\Ad(h): \h_\eta \rightarrow \h_\eta$ stands for the adjoint action of $h\in G_\eta$ on the Lie algebra $\h_\eta$. 
\end{rmk}

\subsection{Prolongation of \texorpdfstring{transverse}{transverse} projective structures}\label{Subsection: Prolongation of transversely projective structures}

Let $\F$ be a smooth codimension $q$ foliation on a complex manifold $X$, and suppose $\F$ admits a smooth transverse projective structure $\mc P$. In this section, we will use the $\PSL(q+1,\C)$-structure on $\prolongPjqtwo$ to construct a natural $\PSL(q+1,\C)$-structure for the foliation $\prolong{\F}{2}$ on $\prolong{X_{\F}}{2}$. We retain the notation of Subsection \ref{S: the second prolongation of the projective space}.

\subsubsection*{Prolongation of the foliated atlas}

Let $\mc P = \{\phi_i: U_i \rightarrow \Pj^q \}$ be a smooth transverse projective structure for the foliation $\F$.  On each sufficiently small open subset of $U_i\cap U_j$, write $\phi_i=L_{g_{ij}}\circ\phi_j$ for a constant $g_{ij}\in G$.

For each $i$, the prolongation of $\phi_i$ is a smooth submersion $\prolong{\phi_i}{2}: \prolong{(U_i)_{\F}}{2} \rightarrow \prolongPjqtwo$ that defines the foliation $\prolong{\F}{2}$.  Equivariance of prolongation gives
\[
 \prolong{\phi_i}{2} = \prolong{L_{g_{ij}}}{2} \circ \prolong{\phi_j}{2}
\]
on the corresponding subset of $\prolong{(U_i)_{\F}}{2}\cap\prolong{(U_j)_{\F}}{2}$. Hence $\mc P$ yields a family of smooth submersions $\prolong{\mc P}{2}=\{\prolong{\phi_i}{2}:\prolong{(U_i)_{\F}}{2}\rightarrow\prolongPjqtwo\}$ defining $\prolong{\F}{2}$, with transition maps locally given by the action of $G$ on $\prolongPjqtwo$.

\subsubsection*{The \texorpdfstring{transverse}{transverse} Lie structure of $\prolong{\F}{2}$}

Let us fix a point $\xi\in G^\circ\subset \prolongPjqtwo$,  and let
\[
    \Omega_{\xi}: T_{\prolongPjqtwo} \to \g \otimes \germe_{\prolongPjqtwo}(D_{\xi})
\]
be the $\PSL(q+1,\C)$-structure of $\prolongPjqtwo$ we defined in Subsection \ref{S: the second prolongation of the projective space}. For every chart $\phi_i:U_i\rightarrow\Pj^q$ of $\mc P$, we consider the meromorphic $\g$-valued 1-form 
\[
\left(\prolong{\phi_i}{2}\right)^*(\Omega_{\xi}): T_{\prolong{(U_i)_{\F}}{2}} \rightarrow \g \otimes \germe_{\prolong{(U_i)_{\F}}{2}}((\prolong{\phi_i}{2})^*D_{\xi}) .
\]
Note that, while $\Omega_\xi$ is a rational $\g$-valued 1-form, the respective pullbacks are only meromorphic, since the atlas $\mc P$ does not, in general, consist of algebraic submersions.  On each sufficiently small intersection as above, we have
\begin{align*}
    \left(\prolong{\phi_i}{2}\right)^*(\Omega_{\xi})  & = \left(\prolong{\phi_j}{2}\right)^* \circ \left(\prolong{L_{g_{ij}}}{2}\right)^*(\Omega_{\xi})  \\
    & = \left(\prolong{\phi_j}{2}\right)^*(\Omega_{\xi}),
\end{align*}
because, by \eqref{Eq: psl structure is invariant by the action}, $\Omega_{\xi}$ is invariant under the action of $G$. Therefore, from the transverse projective structure $\mc P$ we construct a flat meromorphic $\g$-valued 1-form
\begin{equation}
    \prolong{\Omega_{\mc P,\xi}}{2}: T_{\prolong{X_{\F}}{2}} \rightarrow \g \otimes \germe_{\prolong{X_{\F}}{2}}(D).
\end{equation}

\begin{rmk} The $\g$-valued 1-form $\prolong{\Omega_{\mc P,\xi}}{2}$ depends on the point $\xi\in G^\circ$ chosen at the start of the construction. Nevertheless, the transverse $\PSL(q+1,\C)$-structures we obtain are equivalent. Indeed, let $\xi'\in G^\circ$ be a different point, and let $g\in G$ such that $\xi' = g \cdot \xi$. By Proposition \ref{P: rational g valued 1-form depends on q}, we have that
\[
    \prolong{\Omega_{\mc P,\xi'}}{2} = \Ad(g) \circ \prolong{\Omega_{\mc P,\xi}}{2},
\]
and therefore both $\prolong{\Omega_{\mc P,\xi}}{2}$ and $\prolong{\Omega_{\mc P,\xi'}}{2}$ define, up to automorphisms of $G$, the same singular transverse $\PSL(q+1,\C)$-structure for the foliation $\prolong{\F}{2}$. 
\end{rmk}

In view of the remark above, we suppress $\xi$ from the notation and write $\prolong{\Omega_{\mc P}}{2}$ for $\prolong{\Omega_{\mc P,\xi}}{2}$. We will call it the \emph{prolongation of the transverse projective structure $\mc P$}. For each $i$, let $V_i=(\prolong{\phi_i}{2})^{-1}(G^\circ)\subset\prolong{(U_i)_{\F}}{2}$. Since $\prolong{\phi_i}{2}$ is a submersion and the transition maps preserve $G^\circ$, the sets $V_i$ glue to a dense open subset $Y^\circ$ of $Y=\prolong{X_{\F}}{2}$. By the description of $G^\circ$ in Remark \ref{rmk: complement of G circ}, $Y\setminus Y^\circ$ is an analytic divisor given locally by the pullbacks of the divisor $\{\Delta=0\}$ by the maps $\prolong{\phi_i}{2}$. The polar divisor $D$ of $\prolong{\Omega_{\mc P}}{2}$ is supported on this divisor.

\begin{lemma}\label{L: psl structure of the prolongation} Let $\F$ be a codimension $q$ smooth foliation on a complex manifold $X$. Let $\mc P$ be a transverse projective structure for $\F$, and let $\prolong{\Omega_{\mc P}}{2}$ be the prolongation of $\mc P$. Then:
\begin{enumerate}[label = (\roman*)]
    \item\label{I: psl structure of the prolongation restricted to the relative tangent bundle} The poles $D$ of $\prolong{\Omega_{\mc P}}{2}$ are transverse to the fibration $\pi: \prolong{X_{\F}}{2}\rightarrow X$, and 
    \[
    \prolong{\Omega_{\mc P}}{2}\left(T_{\prolong{X_{\F}}{2}/X} \right) \subset \h \otimes \germe_{\prolong{X_{\F}}{2}}(D),
    \]
    where $\h \subset \psl(q+1,\C)$ is the Lie algebra of  $H=G_\eta \subset \PSL(q+1,\C)$, the isotropy subgroup of some $\eta\in \Pj^q$;
    \item\label{I: primitives of the psl structure of the prolongation} every local primitive $\Phi: U\rightarrow\PSL(q+1,\C)$ of $\prolong{\Omega_{\mc P}}{2}$, on a sufficiently small connected open subset where the form is holomorphic, induces a submersion $\phi:\pi(U)\rightarrow\Pj^q$ belonging to the maximal projective atlas generated by $\mc P$.
\end{enumerate}
\end{lemma}
\begin{proof}
Item \ref{I: psl structure of the prolongation restricted to the relative tangent bundle} follows from Proposition \ref{P: psl structure of the prolongation of the projective space, restriction to the relative tangent bundle}. Following the notation above, for Item \ref{I: primitives of the psl structure of the prolongation}, we define on $V_i$ the map
\[
\Psi_i:=\phi_\xi^{-1}\circ\prolong{\phi_i}{2}:V_i\rightarrow\PSL(q+1,\C).
\]
Since $\Omega_\xi=(\phi_\xi^{-1})^*\Omega_G$ on $G^\circ$, the map $\Psi_i$ is a primitive of $\prolong{\Omega_{\mc P}}{2}$ on $V_i$. Let $\eta=\pi(\xi)$ and let $\rho_\eta:\PSL(q+1,\C)\rightarrow\Pj^q$ be the map $g\mapsto g\cdot\eta$. The following diagram commutes:
\begin{equation}\label{D: primitives of the psl structure}
    \begin{tikzpicture}
        \matrix(m)[matrix of math nodes, column sep = 2em, row sep = 2em]
        {
        V_i & G^\circ & \PSL(q+1,\C) \\
        \pi(V_i) & \Pj^q \\
        };
        \path[->]
        (m-1-1) edge node[above]{$\prolong{\phi_i}{2}$} (m-1-2) edge node[left]{$\pi$} (m-2-1)
        (m-1-2) edge node[right]{$\pi$} (m-2-2) edge node[above]{$\phi_\xi^{-1}$} (m-1-3)
        (m-1-3) edge node[right]{$\rho_\eta$} (m-2-2)
        (m-2-1) edge node[below]{$\phi_i$} (m-2-2)
        ;
    \end{tikzpicture}.
\end{equation}
In particular, $\rho_\eta\circ\Psi_i=\phi_i\circ\pi$ on $V_i$.

Let $\Phi:U\rightarrow\PSL(q+1,\C)$ be any local primitive of $\prolong{\Omega_{\mc P}}{2}$ and let $i$ such that $V_i\cap U\neq \emptyset$. Shrinking $U$ if necessary, we assume that the form is holomorphic on $U$, $U\subset\prolong{(U_i)_{\F}}{2}$ and that $U\cap V_i$ is connected. Since both $\Phi$ and $\Psi$ are primitives of $\prolong{\Omega_{\mc P}}{2}$ on $U\cap V_i$, it follows that $\Phi=L_g\circ\Psi_i$ for some $g\in\PSL(q+1,\C)$ (see \cite[Theorem 5.2]{sharpe-97-zbMATH00914851}). Hence
\[
\rho_\eta\circ\Phi=L_g\circ\phi_i\circ\pi.
\]
Thus $\Phi$ induces the map $L_g\circ\phi_i:\pi(U) \rightarrow \Pj^q$, which is a chart of the maximal atlas generated by $\mc P$.
\end{proof}

\begin{rmk} Item \ref{I: psl structure of the prolongation restricted to the relative tangent bundle} alone is enough to conclude that every primitive $\Phi: U \rightarrow \PSL(q+1,\C)$ respects the fibrations, and that the set of induced maps $\phi: \pi(U) \rightarrow \Pj^q$ defines a transverse projective structure $\mc P'$ for $\F$. Hence, Lemma \ref{L: psl structure of the prolongation} states that this transverse projective structure $\mc P'$ is equivalent to the original transverse projective structure $\mc P$.  
\end{rmk}

\begin{thm}\label{T: prolongation of transversely projective structure}
    Let $\F$ be a codimension $q$ smooth foliation on a complex manifold $X$. Let $\mc P$ be a transverse projective structure for $\F$, and let $\prolong{\Omega_{\mc P}}{2}$ be the prolongation of $\mc P$. Then, for every meromorphic section $\sigma:X\dashrightarrow\prolong{X_{\F}}{2}$ that generically takes values in $Y^\circ$ on each connected component of $X$, the pullback $\sigma^*\prolong{\Omega_{\mc P}}{2}$ defines a singular transverse projective structure for $\F$ compatible with $\mc P$.
\end{thm}

\begin{proof} 
Let $X^\circ\subset X$ be the dense open set where $\sigma$ is holomorphic and takes values in $Y^\circ$, and set $\alpha=\sigma^*\prolong{\Omega_{\mc P}}{2}$ on $X^\circ$.  Since $\sigma$ is a meromorphic section and $\prolong{\Omega_{\mc P}}{2}$ is rational in the coordinates of the fibers, it follows that $\alpha$ extends meromorphically to $X$.   Let $D_\alpha=(\alpha)_\infty$ be its polar divisor, including the poles introduced by $\sigma$. Thus $\alpha$ defines a morphism $T_X\rightarrow\g\otimes\germe_X(D_\alpha)$, and its quotient defines a morphism
\[
\bar\alpha:T_X\rightarrow(\g/\h)\otimes\germe_X(D_\alpha).
\]
The Maurer--Cartan equation holds on $X^\circ$ and therefore holds as an identity of meromorphic forms on $X$. It remains to show that $\ker\bar\alpha=T_{\F}$.

Set $Y=\prolong{X_{\F}}{2}$ and $\Omega=\prolong{\Omega_{\mc P}}{2}$. By Lemma \ref{L: psl structure of the prolongation}, Item \ref{I: psl structure of the prolongation restricted to the relative tangent bundle}, the quotient of $\Omega$ by $\h$ vanishes on $T_{Y/X}$ and therefore induces the morphism
\[
\Omega':\pi^*T_X \rightarrow (\g/\h) \otimes \germe_Y(D).
\]
On $Y^\circ$, we have $\ker\Omega=T_{\prolong{\F}{2}}$, the horizontal lift of $\pi^*T_{\F}$ induced by $\nabla_B^1$. Moreover, since $\operatorname{rank}T_{Y/X}=\dim\h=q^2+q$, Item \ref{I: psl structure of the prolongation restricted to the relative tangent bundle} implies that $\Omega|_{T_{Y/X}}$ is generically an isomorphism onto $\h$. By \eqref{eq: diagram of the lifting}, it follows that 
\[
	\ker(\Omega'|_{Y^\circ})=\pi^*T_{\F}|_{Y^\circ}.
\]
Pulling back by $\sigma$ on $X^\circ$, we obtain $\bar\alpha|_{X^\circ}=\sigma^*\Omega'$ and $\ker(\bar\alpha|_{X^\circ})=T_{\F}|_{X^\circ}$. Since $\bar\alpha$ vanishes on $T_{\F}$ over a dense open set, it vanishes on all of $X$. Moreover, the induced morphism from the locally free sheaf $N_{\F}=T_X/T_{\F}$ is generically injective and hence injective as a morphism of sheaves. Therefore $\ker\bar\alpha=T_{\F}$, and $\alpha$ defines a singular transverse projective structure for $\F$.

It remains to verify that this structure is compatible with $\mc P$. Let $V\subset X^\circ$ be a sufficiently small connected open set where $\alpha$ is holomorphic, and let $\Psi:V\rightarrow\PSL(q+1,\C)$ be any primitive of $\alpha$. Let $\Phi:U\rightarrow \PSL(q+1,\C)$ be a primitive of $\Omega$ as in Lemma \ref{L: psl structure of the prolongation}. Then $\Phi \circ \sigma$ is also a primitive of $\alpha$ defined on $U$, and hence $\Psi=L_g\circ\Phi\circ\sigma$ for some $g\in\PSL(q+1,\C)$. If $\phi:U\rightarrow \Pj^q$ is the chart of $\mc P$ induced by $\Phi$, as in Item \ref{I: primitives of the psl structure of the prolongation} of the same lemma, it follows that
\[
\rho_{\eta} \circ \Psi = L_g \circ \phi.
\]
Therefore, the singular transverse projective structure defined by $\alpha$ is compatible with $\mc P$.
\end{proof}

\begin{rmk}
If $X$ is a smooth algebraic variety and $\F$ is algebraic, a meromorphic section satisfying the hypothesis of Theorem \ref{T: prolongation of transversely projective structure} exists. Indeed, in this case $\jetXF^1(\nabla_B)$ is an algebraic vector bundle, and it is enough to consider a generic rational section of it.
\end{rmk}

\subsection{Proof of Lemma \ref{L: trivial isotropy}}\label{S: proof of the action}

Since the projection $\pi:\prolongPjntwo \rightarrow \Pj^n$ is $G$-equivariant, it follows that $G_\xi \subset G_{\pi(\xi)}$ for every $\xi\in \prolongPjntwo$.  Hence, instead of considering the action of $G$ on $\prolongPjntwo$, we can consider the action of the isotropy group $G_\eta\subset G$ of $\eta\in \Pj^n$ over the fiber $\prolongPjntwo_\eta \simeq \C^{n^2 + n}$. We want to conclude that for a generic $\xi\in \prolongPjntwo_\eta$, the isotropy group $G_\xi \subset G_\eta$ is trivial. In order to do that, we will first need to describe the action of $G_\eta$ on $\prolongPjntwo$ in coordinates.

\subsubsection*{Describing the action of $G_\eta$ on $\Pj^n$ in coordinates}

Let $(x_0:x_1: \cdots :x_n)$ be homogeneous coordinates on $\Pj^n$ and let $M=(m_{ij})\in \GL(n+1,\C)$ be a matrix representing the action of $g\in G$ on $\Pj^n$, that is, the action of $g$ on $\Pj^n$ is given by
\[
L_g(x_0: \cdots:x_i: \cdots :x_n) = \left(\sum_{j=0}^n m_{0j} \cdot x_j: \cdots: \sum_{j=0}^n m_{ij}\cdot x_j: \cdots \sum_{j=0}^n m_{nj} \cdot x_j\right).
\]
Without loss of generality, let us suppose that $\eta=(1:0: \cdots:0)$, and consider the affine coordinates $\{y_i =x_i/x_0\}$ on $U= \{x_0\neq 0 \}$. For every $g\in G_\eta$, the automorphism $L_g: \Pj^n \rightarrow \Pj^n$ induces a map $L_g: (\C^n,0) \rightarrow (\C^n,0)$ with respect to the affine coordinates $\mbf y = (y_1,\ldots,y_n)$. Let us describe $L_g$ and its derivatives in these coordinates.

Since $m_{00} \neq 0$ for every $g\in G_\eta$, we may assume $m_{00}=1$. Moreover, we have $G_\eta = \{g \in G; m_{i0}=0, 1\le i \le n \}$, and hence, for every $g\in G_\eta$, we have that
\[
L_g(\mbf y) = \left(\frac{\sum_{j=1}^n m_{1j} \cdot y_j}{1 + \sum_{j=1}^n m_{0j} \cdot y_j}, \ldots,\frac{\sum_{j=1}^n m_{ij} \cdot y_j}{1 + \sum_{j=1}^n m_{0j} \cdot y_j}, \ldots, \frac{\sum_{j=1}^n m_{nj} \cdot y_j}{1 + \sum_{j=1}^n m_{0j} \cdot y_j}  \right).
\]
Expanding in a power series, we can formally write
\begin{align*}
    \nonumber g_i(\mbf y) & := \frac{\sum_{j=1}^n m_{ij} \cdot y_j}{1 + \sum_{j=1}^n m_{0j} \cdot y_j} \\
    \label{Eq: power series expression of the action of psl on the projective space} & = \sum_{j=1}^n m_{ij} \cdot y_j  - \sum_{1\le j_1,j_2\le n} \left(m_{ij_1} \cdot m_{0j_2}\right) \cdot y_{j_1} \cdot y_{j_2} + \mathrm{h.o.t.},
\end{align*}
for $1\le i \le n$. In particular, it follows that 
\begin{equation}\label{Eq: first derivatives of the action of psl on the projective space}
    \frac{\de g_i}{\de y_j}(0)= m_{ij}, 1\le i,j \le n,
\end{equation}
and
\begin{equation}\label{Eq: second derivatives of the action of psl on the projective space}
\frac{\de^2 g_i}{ \de y_{j_1}\de y_{j_2} }(0) = - m_{ij_1} \cdot m_{0j_2} - m_{ij_2} \cdot m_{0j_1}, 1\le i,j_1,j_2 \le n.
\end{equation}

\subsubsection*{Describing the action of $G_\eta$ on $\prolongPjntwo$ in coordinates}

Keeping the notation above, let $\xi\in \prolongPjntwo$ satisfy $\pi(\xi) = \eta = (0,\ldots, 0)$. Let us also introduce the following notation: for every $g\in G_\eta$, represented by the matrix $M = (m_{ij}) \in \GL(n+1,\C)$ with $m_{00}=1$, the inverse $g^{-1} \in G_\eta$ will be represented by the matrix $M^{-1} = (m_{ij}^{-1})$.

Let us consider the natural system of coordinates on $\prolongPjntwo$ $(\mbf y, \mbf z, \mbf w)$ described in Example \ref{Ex: example of the first prolongation of the pushfoward}. Restricting to the fiber over $\eta$, where $\mbf y=0$, the calculations in this example give
\[
\prolong{L_g}{2}(0, \mbf z, \mbf w) = (0,\ldots, (\prolong{L_g}{2})_j(0, \mbf z), \ldots, (\prolong{L_g}{2})_{j_1j_2}(0, \mbf z,\mbf w)),
\]
where, by \eqref{Eq: first derivatives of the action of psl on the projective space} and \eqref{Eq: second derivatives of the action of psl on the projective space},
\[
(\prolong{L_g}{2})_j(0, \mbf z) = \sum_{i=1}^n m_{ji} \cdot z_i,
\]
and
\begin{align*}
    (\prolong{L_g}{2})_{j_1j_2}(0, \mbf z,\mbf w) & = \sum_{i_1,i_2=1}^n m^{-1}_{i_1j_1} \cdot m_{j_2i_2}\cdot w_{i_1i_2} \\
    &  -\sum_{i=1}^n \sum_{k=1}^n m^{-1}_{kj_1} \cdot (m_{j_2k}\cdot m_{0i} + m_{j_2i} \cdot m_{0k}) \cdot z_i \\
    & = \sum_{i_1,i_2=1}^n m^{-1}_{i_1j_1} \cdot m_{j_2i_2}\cdot w_{i_1i_2} - \left(\sum_{k=1}^n m_{kj_1}^{-1}\cdot m_{j_2k}\right)\cdot \left( \sum_{i=1}^n m_{0i} \cdot z_i\right) \\ 
        & - \left(\sum_{k=1}^n m^{-1}_{kj_1}\cdot m_{0k}\right)\left(\sum_{i=1}^n m_{j_2i} \cdot z_i\right)  \\
        & = \sum_{i_1,i_2=1}^n m^{-1}_{i_1j_1} \cdot m_{j_2i_2}\cdot w_{i_1i_2} - \delta_{j_1,j_2} \cdot \left(\sum_{i=1}^n m_{0i} \cdot z_i\right)  \\
        & - \left(\sum_{k=1}^n m^{-1}_{kj_1}\cdot m_{0k}\right)\left(\sum_{i=1}^n m_{j_2i} \cdot z_i\right).
\end{align*}

Let us describe the above expressions using matrices. Let us denote the matrix $(m_{ij})_{1\le i,j \le n}$ by $A$, and the column $(m_{01}, \ldots, m_{0n})^T$ by $B$. Let us also consider the coordinates $\mbf z$ as a vector $Z$, and the coordinates $\mbf w$ as a matrix $W$. With this notation, the action $\prolong{L_g}{2}$ is given by 
\[
    M \cdot (Z,W) = \left(A \cdot Z,(A^T)^{-1} \cdot W \cdot  A^T - B^T\cdot Z\cdot \mathrm{Id} - (A^T)^{-1}\cdot B  \cdot (A\cdot Z)^T \right),
\]
and finding $g\in G_\eta$ such that $g \cdot \xi = \xi$ is equivalent to solving the system 
\begin{equation}\label{Eq: proof of the action, system determining the isotropy group before simplification}
    \begin{cases}
        A\cdot Z = Z, \\
        (A^T)^{-1} \cdot W \cdot  A^T - B^T\cdot Z\cdot \mathrm{Id} - (A^T)^{-1}\cdot B  \cdot (A\cdot Z)^T = W.
    \end{cases}
\end{equation}

\subsubsection*{Simplifying the System \eqref{Eq: proof of the action, system determining the isotropy group before simplification}} 

Observe that it is not necessary to solve the system \eqref{Eq: proof of the action, system determining the isotropy group before simplification}; instead, it is enough to verify that for a generic point $\xi\in \prolongPjntwo_\eta$, the system does not have any non-trivial solution. The next claim shows how we simplify the system.   

\begin{claim}\label{Claim: proof of the action 1}
    Let $M=(A,B)$ be a solution of the system \eqref{Eq: proof of the action, system determining the isotropy group before simplification}. Then $B^T \cdot Z =0$.
\end{claim}
\begin{proof}
    The claim follows by taking traces. Indeed,
    \begin{align*}
        \tr(W) & = \tr((A^T)^{-1} \cdot W \cdot  A^T) - \tr(B^T\cdot Z\cdot \mathrm{Id}) - \tr ((A^T)^{-1}\cdot B  \cdot (A\cdot Z)^T) \\
        & = \tr(W) - n \cdot \tr(B^T \cdot Z) - \tr (B \cdot Z^T) \\
        & = \tr(W) - (n+1) \cdot B^T \cdot Z,
    \end{align*}
    and thus $B^T \cdot Z=0$. 
\end{proof}

For every $\xi=(Z,W)\in \prolongPjntwo_\eta$, let us consider the system of equations
\begin{equation}\label{Eq: proof of the action, system determining the isotropy group}
    \begin{cases}
        A\cdot Z = Z, \\
        B\cdot Z^T = W \cdot A^T - A^T \cdot W.
    \end{cases}
\end{equation}
By the claim above, if the system \eqref{Eq: proof of the action, system determining the isotropy group} admits no non-trivial solutions, the same is true for the  system \eqref{Eq: proof of the action, system determining the isotropy group before simplification}. That is, if for a given $\xi=(Z,W)$ the only solution of the  system \eqref{Eq: proof of the action, system determining the isotropy group} is $A=\Id$ and $B=0$, then $G_\xi$ is trivial.

Our problem now is to prove that for a generic $\xi=(Z,W)$, the only solution of the system \eqref{Eq: proof of the action, system determining the isotropy group} is $(\Id,0)$. Consider the row vectors
\[
Z^T,\ Z^TW,\ \ldots,\ Z^TW^{n-1}.
\]
The condition that these vectors form a basis of $(\C^n)^*$ defines a Zariski open subset of $\C^n\times\mathrm{M}(n,\C)$. This subset is nonempty: take $W$ diagonal with distinct eigenvalues and take every coordinate of $Z$ to be nonzero. The resulting matrix of rows is a Vandermonde matrix multiplied by an invertible diagonal matrix.

Let $(Z,W)$ belong to this open subset, and suppose that $(A,B)$ solves the System \eqref{Eq: proof of the action, system determining the isotropy group}. Multiplying its second equation by $W^k$ and taking the trace, we obtain, for $0\le k<n$,
\[
Z^TW^kB=\tr(W^kBZ^T)=\tr(W^{k+1}A^T-W^kA^TW)=0.
\]
Since the rows $Z^TW^k$ form a basis, it follows that $B=0$. Hence $W$ commutes with $A^T$. The first equation gives $Z^TA^T=Z^T$, and therefore
\[
Z^TW^kA^T=Z^TA^TW^k=Z^TW^k,\ 0\le k<n.
\]
Again, as these rows form a basis, we conclude that $A^T=\Id$. Thus the only solution is $(\Id,0)$, so $G_\xi$ is trivial for generic $\xi$. This concludes the proof. \qed

\bibliographystyle{amsplain}
\bibliography{references}

\end{document}